\documentclass[11pt]{amsart}
\usepackage[utf8x]{inputenc}
\usepackage[english]{babel}
\usepackage[T1]{fontenc} 
 
\usepackage{graphicx}
\usepackage{caption}
\usepackage{subcaption}
\usepackage{amssymb,amsmath}
\usepackage[hidelinks]{hyperref}
\usepackage{indentfirst}
\usepackage{enumerate,amsmath,amssymb, mathrsfs,mathtools}
\usepackage{appendix}
\usepackage{latexsym}
\usepackage{url}
\usepackage{color}
\usepackage{accents}
\usepackage{setspace}
\usepackage{pdfpages}
\usepackage{amsrefs}
\usepackage{stmaryrd}
\usepackage{stringstrings}
\usepackage[margin=2.5cm]{geometry}
\allowdisplaybreaks

\newcommand{\hcirc}{\accentset{\circ}{h}}
\newcommand{\hbarcirc}{\accentset{\circ}{\bar{h}}}

\newtheorem{prop}{Proposition}
\newtheorem{thm}[prop]{Theorem}
\newtheorem{lem}[prop]{Lemma}
\newtheorem{coro}[prop]{Corollary}

\newtheorem{rema}[prop]{Remark}

\title[Foliations by stable constant mean curvature spheres]{Foliations of asymptotically flat manifolds by stable constant mean curvature spheres}
\author{Michael Eichmair}
\address{
	\textnormal{Michael Eichmair \newline  \indent
		University of Vienna \newline \indent
		Faculty of Mathematics  \newline \indent
		Oskar-Morgenstern-Platz 1 \newline \indent
		1090 Vienna, 	Austria  \newline\indent 
		\href{https://orcid.org/0000-0001-7993-9536}{https://orcid.org/0000-0001-7993-9536} \newline\indent	
		\href{mailto:michael.eichmair@univie.ac.at}{michael.eichmair@univie.ac.at}}
}

\author{Thomas Koerber}
\address{\textnormal{Thomas Koerber  \newline \indent
		University of Vienna \newline \indent
		Faculty of Mathematics  \newline \indent
		Oskar-Morgenstern-Platz 1 \newline \indent 1090 Vienna,	Austria \newline\indent 
		\href{https://orcid.org/0000-0003-1676-0824}{https://orcid.org/0000-0003-1676-0824} \newline \indent
		\href{mailto:thomas.koerber@univie.ac.at}{thomas.koerber@univie.ac.at}}
}

\BibSpec{article}{%
	+{}{\PrintAuthors} {author}
	+{,}{ } {title}
	+{, }{\textit } {journal}
	+{}{ \parenthesize} {date}
	+{,  }{no. } {volume}
	+{,}{ } {pages}
	+{,}{ } {note}
}

\ExplSyntaxOn

\ExplSyntaxOff
\BibSpec{book}{%
	+{}{\PrintAuthors}  {author}
	+{. }{}{title}
	+{,}{ }{series}
	+{,}{ vol.~}{volume}
	+{, }{\textit}{publisher}
	+{, }{}{date}
	+{. }{ } {note}
}
\BibSpec{collection.article}{%
	+{}{\PrintAuthors}{author}
	+{, }{}{title}
	+{, }{\textit}{booktitle}
	+{, }{ \DashPages}{pages}
	+{,}{ }{series}
	+{, }{}{volume}
	+{, }{\textit}{publisher}
	+{,}{ }{date}
}

\begin{document}
\date{\today}
\onehalfspacing

\begin{abstract}

Let $(M,g)$ be an asymptotically flat Riemannian manifold of dimension $n\geq 3$ with positive mass. We give a short proof based on Lyapunov-Schmidt reduction of the existence of an asymptotic foliation of $(M, g)$ by stable constant mean curvature spheres. Moreover, we show that the geometric center of mass of the foliation agrees with the Hamiltonian center of mass of $(M,g)$. In dimension $n = 3$, these results were shown previously by C.~Nerz using a different approach. In the case where $n=3$ and the scalar curvature of $(M, g)$ is nonnegative, we prove that the leaves of the asymptotic foliation are the only large stable constant mean curvature spheres that enclose the center of $(M, g)$. This was shown previously under more restrictive decay assumptions and using a different method by S.~Ma.

\end{abstract}
\maketitle	
	\section{Introduction}
		Let $(M,g)$ be an asymptotically flat Riemannian manifold of dimension $n\geq 3$. We refer to Appendix \ref{af manifolds appendix} for  definitions and conventions related to such manifolds that are used in this paper.  \\ \indent 
	Asymptotically flat manifolds arise as initial data for the Einstein field equations modeling isolated gravitational systems. They have been studied extensively in the context of mathematical relativity. It has been established that  the asymptotic geometry of such initial data is intricately tied to the geometry of large stable constant mean curvature spheres in $(M, g)$. \\ \indent 
	The goal of this paper is to provide short, conceptually simple new proofs of some of these results that work for general asymptotically flat Riemannian manifolds in all  dimensions. To describe these results and related work, recall from, e.g.,~\cite{ADM} that the mass of $(M,g)$ is given by
	\begin{align} \label{mass}
	m=\lim_{\lambda\to\infty}\frac{1}{2\,n\,(n-1)\,\omega_n}\,\lambda^{-1}\int_{S_\lambda(0)}\sum_{i,\,j=1}^n\,x^i\,[(\partial_jg)(e_i,e_j)-(\partial_i g)(e_j,e_j)]\,\mathrm{d}\bar\mu
	\end{align}
	where the computation is carried out in an asymptotically flat chart \eqref{asymptotically flat} of $(M,g)$, $e_1,\ldots,e_n$ is the standard basis of $\mathbb{R}^n$, and  $\omega_n$ is the Euclidean volume of an $n$-dimensional Euclidean unit ball. 
	The bar indicates that a geometric quantity is computed with respect to the Euclidean background metric $\bar g$.  Assume that $m\neq 0$.  The Hamiltonian center of mass of $(M,g)$ is defined by $C=(C^1,\ldots,\,C^n)$ where
	\begin{equation}
	\begin{aligned}
	C^\ell =\,&\lim_{\lambda\to\infty}\frac{1}{2\,n\,(n-1)\,\omega_n\,m}\,\lambda^{-1}\,\int_{S_\lambda(0)}\sum_{i,\,j=1}^nx^\ell\,x^j\,\big[(\partial_ig)(e_i,e_j)-(\partial_jg)(e_i,e_i)\big]
	\\&\qquad\qquad\quad \qquad\qquad\qquad\qquad -\sum_{i=1}^n\big[x^i\,g(e_i,e_\ell)-x^\ell\,g(e_i,e_i)\big]\,\mathrm{d}\bar\mu
	\end{aligned}
	\label{center of mass}
	\end{equation}
	provided that the limits on the right-hand side exist for $\ell=1,\ldots,\,n$; see \cite{RT}. \\ \indent A two-sided  hypersurface  $\Sigma\subset M$ is said to have constant mean curvature if its scalar mean curvature $H(\Sigma)$ is constant. We survey general properties of such hypersurfaces in Appendix \ref{CMC appendix}. In the case where $n=3$,   D.~Christodoulou and S.-T.~Yau \cite{ChristodoulouYau} have shown that the quasi-local Hawking mass of a stable constant mean curvature sphere is nonnegative if $(M,g)$ has nonnegative scalar curvature. This observation suggests that the geometry of stable constant mean curvature spheres encodes information about the strength of the gravitational field in the domain they enclose. In their pioneering work \cite{HuiskenYau}, G.~Huisken and S.-T.~Yau have shown that the asymptotic region of an asymptotically flat Riemannian three-manifold that is asymptotic to Schwarzschild
	\eqref{Schwarzschild metric} is foliated by stable constant mean curvature spheres; see also the work \cite{Ye} of R.~Ye in this direction. They have also established a characterization result for the leaves of this foliation, which has been sharpened by J.~Qing and G.~Tian in \cite{QingTian}. We provide further details on these and related contributions in Appendix \ref{Schwarzschild appendix}.

\subsection*{Outline of related results}
 Remarkably, it turns out that many of these results also hold when $(M,g)$ is merely asymptotically flat \eqref{asymptotically flat}. The following result in this direction has been proven by C.~Nerz; see also the discussion on \cite[p.~947]{Ma}. To state the results,  we define the area radius $\lambda(\Sigma)>0$ of a closed hypersurface $\Sigma\subset M$ by $$n\,\omega_n\, \lambda(\Sigma)^{n-1}=|\Sigma|$$ and the inner radius $\rho(\Sigma)$ of such a surface by
 $$
 \rho(\Sigma)=\sup\{r>1: B_{r}\cap\Sigma=\emptyset\}.$$
 We refer to Appendix \ref{af manifolds appendix} for the definition of the sets $B_r$.
 
\begin{thm}[{\cite[Theorems 5.1, 5.2, and 5.3]{Nerz}}]
	Let $(M,g)$ be a connected complete Riemannian three-manifold that is $C^2$-asymptotically flat \eqref{asymptotically flat}  with mass $m\neq 0$. There exists $H_0>0$ and a distinguished family \begin{align} \label{cmc foliation 2} \{\Sigma(H):H\in(0,H_0)\}
	\end{align} 
	 of constant mean curvature spheres \label{existence thm nerz} $\Sigma(H)\subset M$ with mean curvature $H$ that forms a foliation of the complement of a compact subset of  $M$. The spheres $\Sigma(H)$ are stable if and only if $m>0$. \\ \indent  Moreover, given $\delta>0$, there exists $r>1$ such that every stable constant mean curvature sphere $\Sigma\subset M$ that encloses $B_r$ with 
	\begin{align} 
	\delta \,\lambda(\Sigma)<\rho(\Sigma) \label{pinching}
	\end{align} 
	satisfies $\Sigma=\Sigma(H)$ for some $H\in(0,H_0)$.
\end{thm}	

For some settings, a stronger characterization of the leaves of the foliation \eqref{cmc foliation 2} than that stated in Theorem \ref{existence thm nerz} has been obtained. The following global uniqueness result has been established by S.~Ma in \cite{Ma} building on techniques developed in \cite{QingTian}.
\begin{thm}[{\cite[Theorem 1.1]{Ma}}]
	Let $(M,g)$ be a connected complete Riemannian three-manifold that is $C^4$-asymptotically flat of rate $\tau=1$  with mass $m\neq 0$. There exists $r>1$ such that every stable constant mean curvature sphere $\Sigma\subset M$ that encloses $B_r$ belongs to the foliation \label{Ma thm} \eqref{cmc foliation 2}.
\end{thm}
\begin{rema}
	A.~Carlotto and R.~Schoen \cite{SchoenCarlotto} have constructed asymptotically flat Riemannian three-manifolds with positive mass and nonnegative scalar curvature that contain a Euclidean  half-space. Given any number $V>0$, such manifolds contain \label{CS remark} infinitely many stable constant mean curvature spheres whose enclosed volume is equal to $V$. Note that these spheres are neither isoperimetric nor do they enclose the center of $(M,g)$; see \cite[Theorem 1.1]{CESH}.
\end{rema}
\begin{rema}
	The assumption $\tau=1$ in Theorem \ref{Ma thm} is needed throughout \cite[\S4 and \S5]{Ma}.
\end{rema}
In the case where $n=3$, $(M,g)$ is asymptotically flat of rate $\tau\in(1/2,1]$, and $g$ satisfies additional asymptotic assumptions, L.-H.~Huang has proven the following semi-global uniqueness result in \cite{Huang}. We review the so-called Regge-Teitelboim conditions in Appendix \ref{af manifolds appendix}. 
\begin{thm}[{\cite[Theorem 2]{Huang}}]
	Let $(M,g)$ be a connected complete Riemannian three-manifold that is $C^5$-asymptotically flat of rate $\tau\in(1/2,1]$  with mass $m> 0$.		Suppose that $(M,g)$  satisfies the $C^5$-Regge-Teitelboim conditions \eqref{RT} of rate $\tau$. Given $s>1$ with \label{huang uniqueness} \begin{align}
	s<\frac{4+2\,\tau}{5-\tau}, 
	\end{align} 
	there is $r>1$  such that every stable constant mean curvature sphere $\Sigma\subset M$ that encloses $B_r$  with 
	\begin{align} \label{huang pinch}\lambda(\Sigma)<\rho(\Sigma)^s\end{align} 
	belongs to the family \eqref{cmc foliation 2}.
\end{thm}
\begin{rema}
	The  pinching condition \eqref{huang pinch} prevents a sequence $\{\Sigma_i\}_{i=1}^\infty$ of large stable constant mean curvature spheres $\Sigma_i\subset M$ from drifting too quickly with respect to the center of $(M,g)$; see Figure \ref{drifting}. 
\end{rema}
	\begin{figure}\centering
	\includegraphics[width=0.5\linewidth]{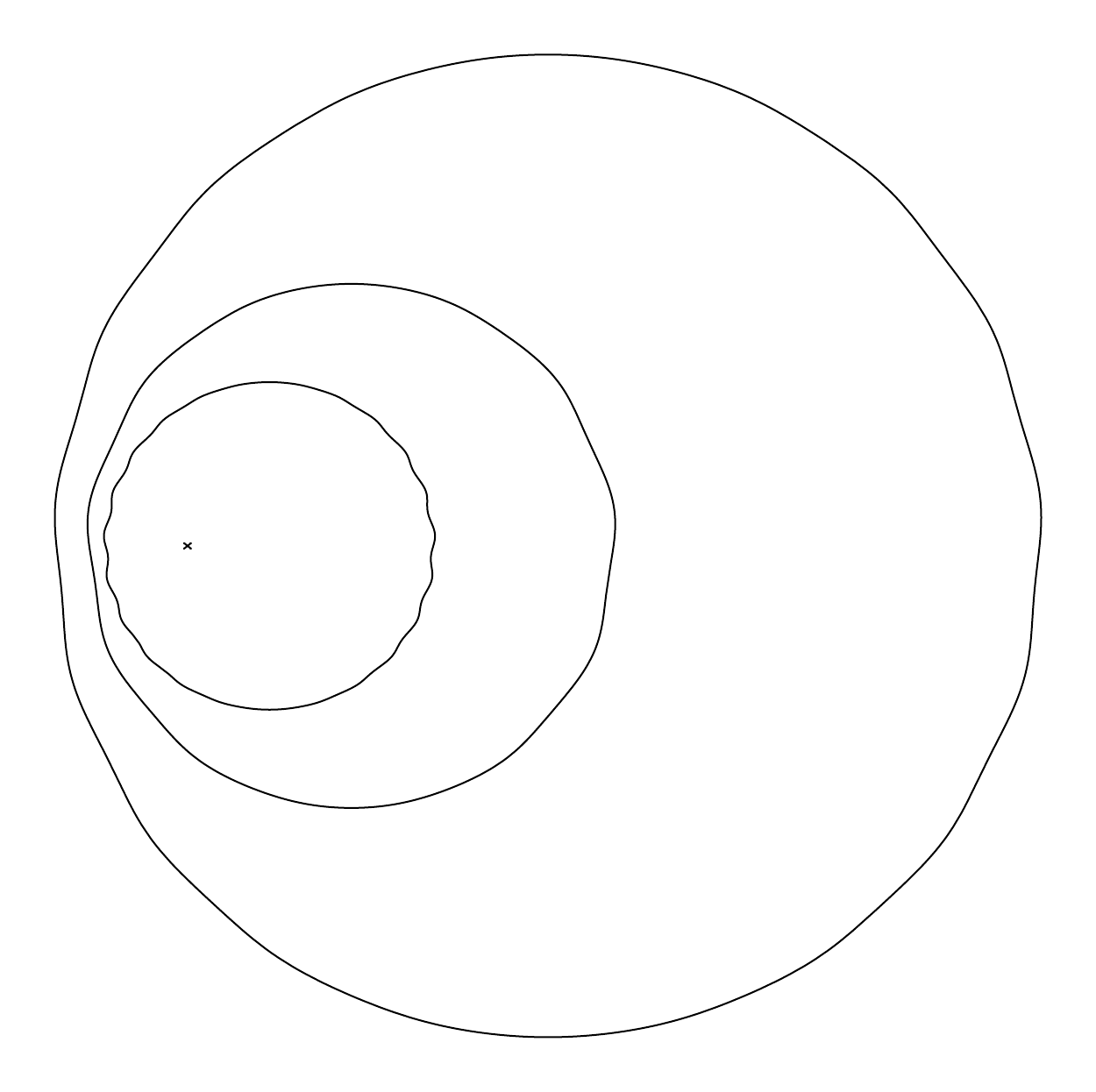}
	\caption{An illustration of the assumption \eqref{huang pinch}. The depicted spheres satisfy \eqref{huang pinch} with $s=2$ but not with $s=3/2$. The cross marks the origin in the asymptotically flat chart.}
	\label{drifting}
\end{figure}
 O.~Chodosh, Y.~Shi, H.~Yu, and the first-named author have shown in \cite{CESH} that the leaves of the foliation \eqref{cmc foliation 2} are globally unique as isoperimetric surfaces provided that $n=3$ and $(M,g)$ has nonnegative scalar curvature; see \cite{yu2020isoperimetry} for an alternative proof by H.~Yu. More precisely, the leaves are the unique surfaces of least area for the amount of volume they enclose. \\ \indent 
The foliation \eqref{cmc foliation 2} leads to a notion of a geometric center of mass $C_{CMC}=(C_{CMC}^1,\,C_{CMC}^2,\,C_{CMC}^3)$ of $(M,g)$ with components given by 
\begin{align} \label{geometri com} 
C_{CMC}^\ell=\lim_{H\to0}|\Sigma(H)|^{-1}\,\int_{\Sigma{(H)}}x^\ell\,\mathrm{d}\mu
\end{align} 
provided the limits on the right-hand side of \eqref{geometri com} exist for $\ell=1,\,2,\,3$. C.~Nerz has shown in \cite{Nerz} that this geometric center of mass agrees with the Hamiltonian center of mass of $(M,g)$ provided $g$ satisfies additional asymptotic assumptions.
\begin{thm}[{\cite[Theorem 6.3]{Nerz}}] 	Let $(M,g)$ be a connected complete Riemannian three-manifold that is $C^3$-asymptotically flat  with mass $m\neq 0$. 	Suppose that $(M,g)$  satisfies the weak $C^2$-Regge-Teitelboim \label{com nerz thm} conditions \eqref{RT Nerz}. The limit in \eqref{center of mass} exists if and only if the limit in \eqref{geometri com} exists, in which case $C=C_{CMC}$.
\end{thm}

	\begin{rema} Let $(M,g)$ be a connected complete Riemannian manifold of dimension $n\geq 3$ that is $C^2$-asymptotically flat  with mass $m\neq 0$. 	If $(M,g)$  satisfies the $C^2$-Regge-Teitelboim conditions, then the Hamiltonian center of mass \eqref{center of mass} of $(M,g)$ is well-defined; see {\cite[Theorem 2.2]{Huang2}} and Proposition \ref{existence COM}.
\end{rema} 
We also mention the important contributions of  L.-H.~Huang \cite{Huang,Huang2} and  of J.~Corvino and H.~Wu \cite{CorwinoWu} in this direction that precede \cite{Nerz}. 

We survey the methods used by L.-H.~Huang, S.~Ma, and C.~Nerz in Appendix \ref{other methods appendix}.
\subsection*{Outline of the results} Our contributions in this paper are threefold. \\ \indent  First, we use the method of Lyapunov-Schmidt reduction to give a conceptually simple and relatively short proof of Theorem \ref{existence thm nerz} that, unlike the approach in \cite{Nerz}, works in all dimensions. 
\begin{thm} \label{existence in all dimensions}
	Let $(M,g)$ be a connected complete Riemannian manifold of dimension $n\geq 3$ that is $C^3$-asymptotically flat with mass $m\neq 0$. There exists $H_0>0$ and a distinguished family \begin{align} \label{cmc foliation 3} \{\Sigma(H):H\in(0,H_0)\} \end{align}
	  of constant mean curvature spheres  $\Sigma(H)\subset M$ with mean curvature $H$ that forms a foliation of the complement of a compact subset of  $M$. The spheres $\Sigma(H)$ are stable if and only if $m>0$. 
\end{thm}	
 As in the case where $n=3$, the foliation \eqref{cmc foliation 3} leads to a notion of a geometric center of mass 
$C_{CMC}=(C_{CMC}^1,\ldots,\,C_{CMC}^n)$ of $(M,g)$ with components given by 
\begin{align} \label{geometri com general} 
	C_{CMC}^\ell=\lim_{H\to0}|\Sigma(H)|^{-1}\,\int_{\Sigma{(H)}}x^\ell\,\mathrm{d}\mu
\end{align} 
provided the limits on the right-hand side of \eqref{geometri com general} exist for $\ell=1,\ldots,\,n$.  We note that, if $3\leq n\leq 7$, our proof of Theorem \ref{existence in all dimensions} also gives local uniqueness results; see Proposition \ref{local uniqueness result 1} and Proposition \ref{local uniqueness result 2}.
\\ \indent 
Second, we also obtain a new proof of Theorem \ref{com nerz thm} that works in all dimensions provided that $g$ satisfies slightly stronger asymptotic assumptions. In the case where $n=3$,  Theorem \ref{com thm} has first been proven by L.-H.~Huang in \cite[Theorem 1]{Huang}.  
\begin{thm} \label{com thm}
Let $(M,g)$ be a connected complete Riemannian manifold of dimension $n\geq 3$ that is $C^3$-asymptotically flat  with  mass $m\neq 0$ and satisfies the $C^2$-Regge-Teitelboim conditions.  Then the limits in \eqref{center of mass} and \eqref{geometri com general} exist and  $C=C_{CMC}$.
\end{thm}

\begin{rema}
	Note that the condition \eqref{RT Nerz} assumed in Theorem \ref{com nerz thm} is weaker than the condition \eqref{RT} assumed in Theorem \ref{com thm}. Analyzing the center of mass \eqref{geometri com} assuming only \eqref{RT Nerz} appears to be beyond our method's reach. In particular, our stronger assumptions are required to control the error terms in the estimate \eqref{first bary}.
\end{rema}
Third, we expand on the work  in the asymptotically Schwarzschild setting of S.~Brendle and the first-named author  \cite{BrendleEichmair} and of O.~Chodosh and the first-named author  \cite{chodoshfar} to investigate the global uniqueness of large stable constant mean curvature spheres in asymptotically flat Riemannian three-manifolds. In the case where the scalar curvature is nonnegative, this approach enables us to extend Theorem \ref{Ma thm} of S.~Ma to all decay rates $\tau\in(1/2,1]$ in the following way. 
\begin{thm}  \label{uniqueness thm}
	Let $(M,g)$ be a connected complete Riemannian three-manifold that is $C^3$-asymptotically flat of rate $\tau\in(1/2,1]$ with $R\geq 0$ and $m>0$. There exists $r>1$ such that every stable constant mean curvature sphere $\Sigma\subset M$ that encloses $B_r$  
	satisfies $\Sigma=\Sigma(H)$ for some $H\in(0,H_0)$.
\end{thm}
\begin{rema}\text{ }
	\begin{enumerate}[(i)]
		\item Unlike in Theorems \ref{existence thm nerz} and \ref{huang uniqueness}, no centering assumption on the surface $\Sigma$ and no asymptotic symmetries on the metric $g$ are required in Theorem \ref{uniqueness thm}.
		\item If $(M,g)$ is not flat $\mathbb{R}^3$, the assumption $m>0$ in Theorem \ref{uniqueness thm} follows from the positive mass theorem; see \cite{SchoenYau} and \cite[Theorem 6.3]{Bartnik}.
		\item 	If  $(M,g)$ contains no properly embedded totally
		geodesic flat planes along which the ambient scalar curvature vanishes, then every stable constant mean curvature sphere $\Sigma\subset M$ with  sufficiently large enclosed volume is disjoint from $B_r$; see \cite[Theorem 1.10]{CCE}.
		\item Note that the short proof of Theorem \ref{uniqueness thm} given in this paper is essentially self-contained except for the use of an estimate on the Hawking mass due to G.~Huisken and T.~Ilmanen \cite{HI} in the proof of Lemma \ref{CE estimate}.
		\item The proof of Theorem \ref{uniqueness thm} is based on curvature estimates that are tailor-made to the case where $n=3$; see Section \ref{curv est section}.
	\end{enumerate}
\end{rema}

In view of Remark \ref{CS remark}, Theorem \ref{uniqueness thm} completes the characterization of large stable constant mean curvature spheres in asymptotically flat Riemannian three-manifolds with nonnegative scalar curvature under general decay assumptions on the metric. In the case where the scalar curvature is allowed to change sign, we obtain the following improvement of the uniqueness results stated in Theorem \ref{existence thm nerz} and Theorem \ref{huang uniqueness}. 
\begin{thm}
	Let $(M,g)$ be a connected complete Riemannian three-manifold that is {$C^2$}-asymptot-\\ically flat of rate $\tau\in(1/2,1)$ with mass $m\neq 0$. Suppose that $R=O(|x|^{-5/2-\tau})$ as $x\to\infty$. \label{uniqueness thm 2}
	Let $s>1$ be such that \begin{align}  \label{s u 2} 
	s<1+\frac{3}{4}\,\frac{2\,\tau-1}{1-\tau}.
	\end{align} 
	There exists $r>1$  such that every stable constant mean curvature sphere $\Sigma\subset M$ that encloses $B_r$ with
	\begin{align} \label{weaker pinch} 
	\lambda(\Sigma)<\rho(\Sigma)^{s}
	\end{align} 	satisfies $\Sigma=\Sigma(H)$ for some $H\in(0,H_0)$.
\end{thm}
\begin{rema}
\text{ }
	\begin{enumerate}[(i)]
		\item The condition \eqref{weaker pinch} is weaker than \eqref{pinching}.
		\item Unlike in Theorem \ref{huang uniqueness}, we impose no further assumptions on the asymptotic symmetries of $g$ in Theorem \ref{uniqueness thm 2}.
		\item The centering assumption \eqref{weaker pinch} is less restrictive than \eqref{huang pinch} if and only if $$ \tau>\frac{\sqrt{553}-17}{12}\approx 0.55.$$
		\item The bound \eqref{s u 2} seems to be the best possible for the method developed in this paper; see \eqref{3 s}.
		\item If $m<0$, Theorem \ref{uniqueness thm 2} and Theorem \ref{existence thm nerz} imply that there are no large stable constant mean curvature spheres in $(M,g)$ that enclose $B_r$ and satisfy \eqref{weaker pinch}.
	\end{enumerate}
\end{rema}
\begin{rema}
	We comment on the regularity assumptions in Theorem \ref{existence in all dimensions} and Theorem \ref{uniqueness thm}. Let $(M,g)$ be a connected complete Riemannian manifold of dimension $n\geq 3$ that is $C^2$-asymptotically flat. 	The arguments in Section \ref{existence section} show there is $\lambda_0>1$ such that for every $\lambda>\lambda_0$ there exists a constant mean curvature sphere $\Sigma(\lambda)$ with area radius $\lambda$.  Moreover, the arguments in Section \ref{curv est section} and Section \ref{uniqueness section} show that, in the case where $n=3$ and $R\geq 0$, given $\delta>0$, there exists $r>1$ such that there is no stable constant mean curvature sphere $\Sigma\subset M$ that encloses $B_r$ with $(1-\delta)\,\lambda(\Sigma)>\rho(\Sigma)$. These results do not require that $R=o(|x|^{-n})$ as $x\to\infty$. This assumption is used only in the estimate \eqref{r estimate intro} to compute the second derivative of the reduced area function \eqref{g intro}.  By contrast, the assumption that $(M,g)$ is $C^3$-asymptotically flat is required to derive the estimate in the second and third line of \eqref{u estimate}. This estimate in turn is essential to conclude that the spheres $\Sigma(\lambda)$ are stable if and only if $m>0$ and also that the family $\{\Sigma(\lambda):\lambda>\lambda_0\}$ forms a foliation of the complement of a compact subset. 
\end{rema}
\subsection*{Outline of our arguments}To prove Theorem \ref{existence in all dimensions} and Theorem \ref{com thm}, we expand upon the method of Lyapunov-Schmidt reduction as applied in \cite{BrendleEichmair,chodoshfar,acws}. 
\\ \indent 
Let $\delta\in(0,1/2)$. We use the implicit function theorem to construct surfaces $\Sigma_{\xi,\lambda}$ as perturbations of the Euclidean coordinate spheres $$S_\lambda(\lambda\,\xi)=\{x\in\mathbb{R}^n:|x-\lambda\,\xi|=\lambda\}$$ where $\xi\in\mathbb{R}^n$ with $|\xi|<1-\delta$ and $\lambda>1$ is large   such that $\operatorname{vol}(\Sigma_{\xi,\lambda})$ does not depend on $\xi$ and such that $\Sigma_{\xi,\lambda}$ is a constant mean curvature sphere  if and only if $\xi$ is a critical point of the function $G_\lambda$ defined by
\begin{align} G_\lambda(\xi)=\lambda^{-1}\,|\Sigma_{\xi,\lambda}|. \label{g intro}
\end{align}
Using an integration by parts  inspired by the arguments in \cite[\S2.1]{chodoshfar}, we show that
\begin{align} \label{G expansion intro} 
G_\lambda(\xi)=G_\lambda(0)+\frac12\,n\,(n-1)\,\omega_n\,m\,|\xi|^2+G_{2,\lambda}(\xi)
\end{align} 
where $G_{2,\lambda}(\xi)=o(1)$ as $\lambda\to\infty$. If $m\neq 0$, it follows that $G_\lambda$ has a unique critical point $\xi(\lambda)$ with $|\xi(\lambda)|<1/2$ as $\lambda\to\infty$. This proves Theorem \ref{existence in all dimensions}. To prove Theorem \ref{com thm}, we observe that
\begin{align*}
\lambda\,\xi(\lambda)=|\Sigma(\lambda)|^{-1}\int_{\Sigma(\lambda)} x^\ell\,\mathrm{d}\mu+o(1)
\end{align*}
if $(M,g)$ satisfies the $C^2$-Regge-Teitelboim conditions where $\Sigma(\lambda)=\Sigma_{\xi(\lambda),\lambda}$. Moreover, we  show that $\lambda\,(\bar DG_{2,\lambda})|_{\xi(\lambda)}$ is essentially proportional to the Hamiltonian center of mass $C$ provided that $\lambda>1$ is sufficiently large and $(M,g)$ satisfies the $C^2$-Regge-Teitelboim conditions.   \\ \indent The proofs of Theorem \ref{uniqueness thm} and Theorem \ref{uniqueness thm 2} are based on curvature estimates and an integration by parts that have been observed and used in a related context  by O.~Chodosh and the first-named author in \cite{chodoshfar,chodosh2017global}. \\ \indent Using an estimate proven by D.~Christodoulou and S.-T.~Yau in \cite{ChristodoulouYau} together with global arguments developed by G.~Huisken and T.~Ilmanen in \cite{HI}, we obtain an improved  $L^2$-estimate for the traceless second fundamental $\hcirc(\Sigma)$ of a stable constant mean curvature sphere $\Sigma\subset M$. Such an estimate has also been used by O.~Chodosh and the first named-author in \cite{chodosh2017global} and previously in the dissertation of O.~Chodosh \cite{chodoshthesis}. 
  This allows us to prove curvature estimates for $\Sigma$ that are slightly stronger than those available in the literature. In particular, these estimates improve when the scalar curvature of $(M,g)$ is nonnegative. We then suppose, for a contradiction, that there exists a sequence $\{\Sigma_i\}_{i=1}^\infty$ of large stable constant mean curvature spheres $\Sigma_i\subset M$ that enclose the center of $(M,g)$ and do not belong to the foliation \eqref{cmc foliation 3}. In light of \eqref{G expansion intro}, we may assume that
\begin{align} \label{slow divergence intro}
\lim_{i\to\infty}\rho(\Sigma_i)=\infty \qquad\text{and}\qquad \lim_{i\to\infty}\lambda(\Sigma_i)^{-1}\,\rho(\Sigma_i)=0. 
\end{align}
Since $H=H(\Sigma_i)$ is constant, we have
\begin{align} \label{incompatible}  
0=\int_{\Sigma_i} H\,g(a,\nu)\,\mathrm{d}\mu-H(\Sigma_i)\,\int_{\Sigma_i}g(a,\nu)\,\mathrm{d}\mu 
\end{align} 
for every $a\in\mathbb{R}^3$ with $|a|=1$. Using integration by parts,  that $\Sigma_i$ encloses the center of $(M,g)$, and \eqref{slow divergence intro},  we show that the right-hand side of \eqref{incompatible} equals $4\,\pi\,m+E_i$ for some error term $E_i$ provided $a\in\mathbb{R}^3$ is chosen appropriately. If $R\geq 0,$ or, alternatively, if $R=O(|x|^{-5/2-\tau})$ as $x\to\infty$ and
$\lambda(\Sigma_i)=o(\rho(\Sigma_i)^s)$ as $i\to\infty$ where $s>1$ satisfies
\eqref{s u 2}, the curvature estimates imply that $E_i=o(1)$. This is incompatible with \eqref{incompatible} so Theorem \ref{uniqueness thm} and Theorem \ref{uniqueness thm 2} follow. 
\subsection*{Acknowledgments} 
The authors thank the anonymous referees for their helpful feedback and for encouraging them to extend Theorem \ref{existence in all dimensions} to all dimensions. 
The authors thank Lan-Hsuan Huang for kindly answering  questions related to her work.  Michael Eichmair acknowledges the support of the START Programme Y963 of the Austrian Science Fund. Thomas Koerber acknowledges the support of the Lise Meitner Programme M3184 of the Austrian Science Fund.
\section{Proof of Theorem \ref{existence in all dimensions}} \label{existence section}
The method of Lyapunov-Schmidt reduction has been used by S.~Brendle and the first-named author in \cite{BrendleEichmair} and by O.~Chodosh and the first-named author in \cite{chodoshfar} to study large  stable constant mean curvature spheres that do not enclose the center of a Riemannian three-manifold that is  asymptotic to Schwarzschild \eqref{Schwarzschild metric}. In \cite{acws}, the authors have used the method of Lyapunov-Schmidt reduction to study so-called large area-constrained Willmore spheres in asymptotically Schwarzschild manifolds. Here, we adapt this approach to study constant mean curvature spheres that enclose the center of a general asymptotically flat Riemannian manifold. \\ \indent  
Let $n\geq 3$. In this section, we assume that $g$ is a Riemannian metric on $\mathbb{R}^n$ whose scalar curvature is integrable with $R=o(|x|^{-n})$ as $x\to\infty$ and that there is $\tau\in((n-2)/2,n-2]$ with  
\begin{align} \label{section 2 decay} 
g=\bar g+\sigma \qquad \text{where} \qquad 
\partial_J \sigma=O\left(|x|^{-\tau-|J|}\right)
\end{align} 
for every multi-index $J$ with $|J|\leq 3$.
 \\ \indent 
Given $\xi\in\mathbb{R}^n$ and $\lambda>1$, we abbreviate
$${S}_{\xi,\lambda}=S_{\lambda}(\lambda\,\xi)=\{x\in\mathbb{R}^n:|x-\lambda\,\xi|=\lambda\}.$$
Given $u\in C^{\infty}(S_{\xi,\lambda})$, we define the map
\begin{align*}
\Phi^u_{\xi,\lambda}:S_{\xi,\lambda}\to \mathbb{R}^n \qquad \text{given by} \qquad  \Phi^u_{\xi,\lambda}(x)=x+u(x)\, (\lambda^{-1}\,x-\xi). 
\end{align*}
We denote by $$\Sigma_{\xi,\lambda}(u)=\Phi^u_{\xi,\lambda}(S_{\xi,\lambda})$$ the Euclidean graph of $u$ over $S_{\xi,\lambda}$. We  identify functions defined on $\Sigma_{\xi,\lambda}(u)$ with functions defined on $S_{\xi,\lambda}$ by tacit precomposition with $\Phi^u_{\xi,\lambda}$; see, e.g., Proposition \ref{cmc ls prop}.  	Moreover, we define the map $$\Theta_{\xi,\lambda}:\mathbb{R}^n\to\mathbb{R}^n\qquad \text{ given by }\qquad  \Theta_{\xi,\lambda}(y)= \lambda\,(\xi+y).$$ Note that $\Theta_{\xi,\lambda}(S_1(0))=S_{\xi,\lambda}$.
\\ \indent Let $\delta\in(0,1/2)$. In the  statement of  Proposition \ref{cmc ls prop}, we use $\Lambda_0(S_{\xi,\lambda})$ and $\Lambda_1(S_{\xi,\lambda})$ to denote the constant functions and the first spherical harmonics viewed as subspaces of $C^{\infty}(S_{\xi,\lambda})$, respectively. We use  $\perp$ to denote the orthogonal complements of these spaces in  $C^{\infty}(S_{\xi,\lambda})$ with respect to the Euclidean $L^2$-inner product. In the estimate \eqref{u estimate}, $\bar D$, the dash, and $\bar\nabla$ denote differentiation with respect to $\xi\in\mathbb{R}^n$, $\lambda>1$, and $x\in S_{\xi,\lambda}$, respectively. 
\begin{prop} \label{cmc ls prop}
	There are constants $\lambda_0>1$ and $\varepsilon>0$ depending on $g$ and $\delta\in(0,1/2)$ such that for every $\xi\in\mathbb{R}^n$ with $|\xi|<1-\delta$ and  $\lambda>\lambda_0$ there exists a function $v_{\xi,\lambda}\in C^{\infty}(S_1(0))$ with the following properties. 
		There holds $v_{\xi,\lambda}\perp \Lambda_1(S_{1}(0))$ and, as $\lambda\to\infty$, 
	\begin{equation} \label{u estimate} 
	\begin{aligned} 
	|v_{\xi,\lambda}|+|\bar\nabla v_{\xi,\lambda}|+|\bar\nabla^2 v_{\xi,\lambda}|=&\,O(\lambda^{-\tau}),\\ 
|\bar Dv_{\xi,\lambda}|+|\bar D\bar\nabla v_{\xi,\lambda}|+|\bar D\nabla ^2v_{\xi,\lambda}|=&\,O(\lambda^{-\tau}), \\
	|v_{\xi,\lambda}'|+|\bar\nabla v_{\xi,\lambda}'|+|\bar\nabla^2 v_{\xi,\lambda}'|=&\,O(\lambda^{-1-\tau})
	\end{aligned}
	\end{equation}  
	 uniformly for all $\xi\in\mathbb{R}^n$ with $|\xi|<1-\delta$. Let $u_{\xi,\lambda}\in C^\infty(S_{\xi,\lambda})$ be given by $$u_{\xi,\lambda}(x)=v_{\xi,\lambda}(\lambda^{-1}x-\xi).$$ The surface
	$$
	\Sigma_{\xi,\lambda}=\Sigma_{\xi,\lambda}(u_{\xi,\lambda})
	$$
satisfies
	\begin{itemize}
		\item[$\circ$] $H\in \Lambda_0(S_{\xi,\lambda})\oplus\Lambda_1(S_{\xi,\lambda})$ and
		\item[$\circ$] $\operatorname{vol}(S_{2\,\lambda}(0))-\operatorname{vol}(\Sigma_{\xi,\lambda})=7\,\omega_n\,\lambda^n$. 
	\end{itemize}
	Moreover, if $\Sigma_{\xi,\lambda}(u)$ with $u\perp \Lambda_1(S_{\xi,\lambda})$ is such that 
	\begin{itemize}
		\item[$\circ$] $H\in \Lambda_0(S_{\xi,\lambda})\oplus\Lambda_1(S_{\xi,\lambda})$,
		\item[$\circ$] $\operatorname{vol}(S_{2\,\lambda}(0))-\operatorname{vol}(\Sigma_{\xi,\lambda}(u))=7\,\omega_n\,\lambda^n$, 
		\item[$\circ$]
		$
		\lambda^{-1}\,|u|+|\bar\nabla u|+\lambda\,|\bar\nabla^2 u|<\varepsilon,
		$
	\end{itemize}
	then $u=u_{\xi,\lambda}$.
\end{prop}
\begin{proof}
	This is similar to, e.g.,~\cite[Proposition 4]{BrendleEichmair}. We include a proof for the reader's convenience.\\ \indent 	
	Let $\alpha>0$ and $\mathcal{G}$ be the space of Riemannian metrics on $U=\{y\in\mathbb{R}^n:1-\delta/2<|y|<2\}$ equipped with the $C^{1,\alpha}$-topology.  The rescaled metric
	$		g_{\xi,\lambda}=\lambda^{-2}\,\Theta_{\xi,\lambda}^*\,g$ 	satisfies
	$$
	||g_{\xi,\lambda}-\bar{g}||_{\mathcal{G}}=O(\lambda^{-\tau}\,|1-|\xi||^{-\tau})=O(\lambda^{-\tau}).
	$$
	Moreover, we have
	$$
	||\bar D g_{\xi,\lambda}||_{C^{1,\alpha}(U)}=O(\lambda^{-\tau})
	$$
	and 
	$$
	||g_{\xi,\lambda}'||_{C^{1,\alpha}(U)}=O(\lambda^{-1-\tau}).
	$$
	\indent Let $k\geq 0$ be an integer. Let  $\Lambda_{0,k}(S_1(0))$ and $\Lambda_{1,k}(S_1(0))$ be the constants and first spherical harmonics viewed as subspaces of $C^{k,\alpha}(S_1(0))$, respectively. We define the smooth map 
	\[
	T:\Lambda_{1,2}(S_1(0))^\perp\times \mathcal{G}\to[\Lambda_{0,0}(S_1(0))\oplus\Lambda_{1,0}(S_1(0))]^\perp \times \mathbb{R}
	\]
	by 
	\[
	T(v,\,g)=\left(\operatorname{proj}_{[\Lambda_{0,0}(S_1(0))\oplus\Lambda_{1,0}(S_1(0))]^\perp} H,\,\operatorname{vol}(S_{2}(0))-\operatorname{vol}(\Sigma_{0,1}(u))\right)
	\]
		where the mean curvature is computed with respect to $\Sigma_{0,1}(v)$ and $g$ and the  volume is computed with respect to  $g$. Using \eqref{mean curvature change}, we see that
		$$
		(D T)|_{(0,\,\bar g)}(v,\,0)=\left(\operatorname{proj}_{[\Lambda_{0,0}(S_1(0))\oplus\Lambda_{1,0}(S_1(0))]^\perp}(-\bar \Delta v-(n-1)\, v),\,-n\,\omega_n\,\operatorname{proj}_{\Lambda_0(S_1(0))} v\right).
		$$
	By Lemma \ref{spherical harmonics lemma 1},	the kernel of the operator
		$$
		-\bar\Delta-(n-1):C^{2,\alpha}(S_1(0))\to C^{0,\alpha}(S_1(0)) 
		$$
		is given by $\Lambda_{1,2}(S_1(0))$. It follows that $$(D T)|_{(0,\,\bar g)}(\,\cdot\,,\,0):\Lambda_{1,2}(S_1(0))^\perp\to [\Lambda_{0,0}(S_1(0))\oplus\Lambda_{1,0}(S_1(0))]^\perp \times \mathbb{R}$$ 
		is an isomorphism.   The assertions follow from this and the implicit function theorem.
\end{proof} 
To capture the variational nature of the constant mean curvature equation on the family of surfaces $\{\Sigma_{\xi,\lambda}:|\xi|<1-\delta\}$ from Proposition \ref{cmc ls prop}, we consider the reduced area function $$G_\lambda:\{\xi\in\mathbb{R}^n:|\xi|<1-\delta\}\to\mathbb{R} \qquad\text{given by} 
\qquad G_\lambda(\xi)=\lambda^{-1}\,|\Sigma_{\xi,\lambda}|.
$$
\begin{lem} \label{variation}
		Let $\delta\in(0,1/2)$, $\lambda>\lambda_0$, and $\xi\in\mathbb{R}^n$ with $|\xi|<1-\delta$. Given  $a\in \mathbb{R}^n$ with $|a|=1$, let $f\in C^\infty(\Sigma_{\xi,\lambda})$ be the normal speed  of the variation  $\{\Sigma_{\xi+s\,a,\lambda}:s\in(-\varepsilon,\varepsilon)\}$ at $s=0$. There holds  
		$$
	f=	\lambda\,g(\nu, a)+g(\nu,{\Theta_{\xi,\lambda}}_*\bar D_a v_{\xi,\lambda}).
		$$
\end{lem}
\begin{proof}
Note that the map $$\Sigma_{\xi,\lambda}\times(-\varepsilon,\varepsilon)\to \mathbb{R}^3\qquad\text{given by}\qquad 
	(x,s)\mapsto x+s\,\lambda\,a+v_{\xi+s\,a,\lambda}(\lambda^{-1}\,x-\xi)\,(\lambda^{-1}\,x-\xi)
	$$ 
	parametrizes the family 	$
	\{\Sigma_{\xi+s\,a,\lambda}:s\in(-\varepsilon,\varepsilon)\}.
	$
The assertion follows.
\end{proof}
\begin{lem}
	Given $\delta\in(0,1/2)$,	there is $\lambda_0>1$ such that for every \label{variational to 3d} $\lambda>\lambda_0$ and $\xi\in\mathbb{R}^n$ with $|\xi|<1-\delta$ the following holds. The sphere $\Sigma_{\xi,\lambda}$ has  constant mean curvature if and only if $\xi$ is a  critical point of $G_\lambda$. 
\end{lem}
\begin{proof}
	This is similar to, e.g., \cite[Lemma 21]{acws}. We include a proof for the reader's convenience. \\ \indent 
	Let $a\in\mathbb{R}^n$ with $|a|=1$ and  $f$ the normal speed of the  variation
	$
	\{\Sigma_{\xi+s\,a,\lambda}:s\in(-\varepsilon,\varepsilon)\}.
	$
	Since the variation is volume preserving,
	$$
	\int_{\Sigma_{\xi,\lambda}} f\,\text{d}\mu=0.
	$$
	\indent Assume that $\xi$ is a critical point of $G_\lambda$. We find
	$$
	0=\lambda^{-1}\,	\int_{\Sigma_{\xi,\lambda}} H(\Sigma_{\xi,\lambda}) \,f\,\text{d}\mu.
	$$
	In conjunction with Lemma \ref{variation}, \eqref{section 2 decay}, \eqref{u estimate},  and Lemma \ref{perturbations},
	$$
	0=\int_{S_{\xi,\lambda}} (H(\Sigma_{\xi,\lambda})-\operatorname{proj}_{\Lambda_0(S_{\xi,\lambda})}H(\Sigma_{\xi,\lambda})) \,(\bar g(a,\bar \nu)+o(1))\,\text{d}\bar \mu.
	$$
	Recall that $H(\Sigma_{\xi,\lambda})\in \Lambda_0(S_{\xi,\lambda})\oplus \Lambda_1(S_{\xi,\lambda}) $. Varying the direction of translation $a$, it follows that $H(\Sigma_{\xi,\lambda})$ is constant provided that $\lambda$ is sufficiently large.\\ \indent  
	Conversely, if $\Sigma_{\xi,\lambda}$ is a constant mean curvature sphere, then
	$$
	\int_{\Sigma_{\xi,\lambda}} H\,f\,\text{d}\mu=H(\Sigma_{\xi,\lambda})\,\int_{\Sigma_{\xi,\lambda}} f\,\text{d}\mu=0.
	$$
	In particular, $\xi$ is a critical point of $G_\lambda$. 
\end{proof}
In the following two lemmas, we compute the asymptotic expansion of $G_\lambda$ as $\lambda\to\infty$. Here and below, we say that an error term $\mathcal{E}=O(\lambda^{-\varepsilon})$, $\varepsilon>0$, may be differentiated with respect to $\xi$ if $\bar D\mathcal{E}=O(\lambda^{-\varepsilon})$. 
\begin{lem}
	Let $a\in\mathbb{R}^n$ with $|a|=1$.
	There holds, \label{G expansion} as $\lambda\to\infty$,
	$$
	 (\bar D_aG_\lambda)|_\xi=\frac12\, \int_{S_{\xi,\lambda}}\bar D_a\bar{\operatorname{tr}}\,\sigma-(\bar D_a\sigma)(\bar\nu,\bar\nu)-(n-1)\,\lambda^{-1}\,\bar{\operatorname{tr}}\,\sigma\,\bar g(a,\bar\nu)\,\mathrm{d}\bar\mu+o(1)
	$$
 uniformly for all $\xi\in\mathbb{R}^n$ with $|\xi|<1-\delta$. This estimate may be differentiated once with respect to $\xi$. 
\end{lem} 
\begin{proof}
Let $f$ be	the normal speed of the variation
	$
	\{\Sigma_{\xi+s\,a,\lambda}:s\in(-\varepsilon,\varepsilon)\}
	$.	
	Using that $\operatorname{vol}(\Sigma_{\xi,\lambda})$ does not depend on $\xi\in\mathbb{R}^n$, we obtain
	\begin{align} \label{-1 0}
	(\bar D_{a}G_\lambda)|_\xi=	\int_{\Sigma_{\xi,\lambda}}[H-(n-1)\,\lambda^{-1}]\,f\,\mathrm{d}\mu.
\end{align}
\indent The estimates below may all be differentiated once with respect to $\xi$. 
Using \eqref{mean curvature change}, \eqref{u estimate}, and Lemma \ref{perturbations}, we obtain
\begin{align} \label{MC expansion} 
H(\Sigma_{\xi,\lambda})=H(S_{\xi,\lambda})+O(\lambda^{-1-\tau})=(n-1)\,\lambda^{-1}+O(\lambda^{-1-\tau}).
\end{align}
In conjunction with Lemma \ref{variation}, \eqref{u estimate}, and \eqref{-1 0}, we find
\begin{align}  \label{-1 1}
(\bar D_{a}G_\lambda)|_\xi=	\int_{\Sigma_{\xi,\lambda}}(H-(n-1)\,\lambda^{-1})\,g(a,\nu)\,\mathrm{d}\mu+o(1).
\end{align} 
 By the first variation formula, 
\begin{equation}  \label{-1 2}
\begin{aligned} 
\int_{\Sigma_{\xi,\lambda}}(H-(n-1)\,\lambda^{-1})\,g(a,\nu)\,\mathrm{d}\mu=\,\int_{\Sigma_{\xi,\lambda}}\operatorname{div}a-g(D_\nu a,\nu)-(n-1)\,\lambda^{-1}\,g(a,\nu)\,\mathrm{d}\mu.
\end{aligned} 
\end{equation} 
 Using \eqref{section 2 decay} and \eqref{u estimate}, we have
$$
(\operatorname{div}a)\circ\Phi^{u_{\xi,\lambda}}_{\xi,\lambda}=\operatorname{div}a+O(\lambda^{-2-\tau}\,|u_{\xi,\lambda}|)=\operatorname{div}a+O(\lambda^{-1-2\,\tau})
$$
and
$$
(\Phi^{u_{\xi,\lambda}}_{\xi,\lambda})^*\mathrm{d}\mu(\Sigma_{\xi,\lambda})=\mathrm{d}\mu(S_{\xi,\lambda})+O(|\bar \nabla u_{\xi,\lambda}|)=\mathrm{d}\mu(S_{\xi,\lambda})+O(\lambda^{-\tau}).
$$
It follows that 
$$
\int_{\Sigma_{\xi,\lambda}}\operatorname{div}a\,\mathrm{d}\mu=\int_{S_{\xi,\lambda}}\operatorname{div}a\,\mathrm{d}\mu+o(1).
$$
Likewise, 
$$
\int_{\Sigma_{\xi,\lambda}}g(D_\nu a,\nu)\,\mathrm{d}\mu=\int_{S_{\xi,\lambda}}g(D_\nu a,\nu)\,\mathrm{d}\mu+o(1)$$
and
$$ 
\int_{\Sigma_{\xi,\lambda}}g(a,\nu)\,\mathrm{d}\mu=\int_{S_{\xi,\lambda}}g(a,\nu)\,\mathrm{d}\mu+O(\lambda^{n-1-2\,\tau}).
$$
In conjunction with \eqref{-1 1} and \eqref{-1 2}, this gives
$$
(\bar D_aG_\lambda)|_{\xi}=\int_{S_{\xi,\lambda}}\operatorname{div}a-g(D_\nu a,\nu)-(n-1)\,\lambda^{-1}\,g(a,\nu)\,\mathrm{d}\mu+o(1).
$$
The assertion now follows from  Lemma \ref{perturbations}.
\end{proof}

\begin{lem}
Let $\delta\in(0,1/2)$. There holds, as $\lambda\to\infty$, \label{g expansion} 
\begin{equation*}
\begin{aligned}
G_\lambda(\xi)=\,&G_\lambda(0)+\frac12\,n\,(n-1)\,\omega_n\,m\,|\xi|^2+o(1), \\
(\bar DG_\lambda)|_\xi=\,&n\,(n-1)\,\omega_n\,m\,\xi+o(1), \text{ and}\\
(\bar D^2G_\lambda)|_\xi=\,&n\,(n-1)\,\omega_n\,m\,\operatorname{Id}+o(1)
\end{aligned}
\end{equation*}
 uniformly for all $\xi\in\mathbb{R}^n$ with $|\xi|<1-\delta$.
\end{lem}
\begin{proof}
	Let $a\in\mathbb{R}^n$ with $|a|=1$. Using Lemma \ref{G expansion} and Lemma \ref{ce int by parts rema}, we have
	\begin{align*} 
	(	\bar D_a G_\lambda)|_\xi=&\,\frac12\,\lambda^{-1} \int_{S_{\xi,\lambda}}\bar g(a,x-\lambda\,\xi)\,\big[\bar D_{\bar\nu}\bar{\operatorname{tr}}\,\sigma- (\bar{\operatorname{div}}\,\sigma)(\bar\nu)\big] +\sigma(\bar\nu,a)-\bar g(a,\bar\nu)\,\bar{\operatorname{tr}}\,\sigma\,\mathrm{d}\bar\mu\\&\qquad+o(1).
	\end{align*} 
	Note that,	by \eqref{scalar curvature},
	 \begin{align*}
	& \bar{\operatorname{div}}\left(\bar g(a,x-\lambda\,\xi)\,[\bar D\,\bar{\operatorname{tr}}\,\sigma-(\bar{\operatorname{div}}\,\sigma)]+\, [\sigma(a,\,\cdot\,)-\bar g(a,\,\cdot\,)\,\bar{\operatorname{tr}}\,\sigma]\right) =-R\,\bar g(a,x-\lambda\,\xi)+O(|x|^{-1-2\,\tau}).
	\end{align*} 

 Using the divergence theorem, 	we find that
		\begin{equation} 
		\begin{aligned} \label{compare} 
	(	\bar D_a G_\lambda)|_\xi
	=&\,\frac12\,\bar g(a,\xi)\int_{S_{2\,\lambda}(0)} (\bar{\operatorname{div}}\,\sigma)(\bar\nu)-\bar D_{\bar\nu}\bar{\operatorname{tr}}\,\sigma\,\mathrm{d}\bar\mu\\
	&\qquad+\frac12\,\lambda^{-1} \int_{S_{2\,\lambda}(0)}\bar g(a,x)\,\big[\bar D_{\bar\nu}\bar{\operatorname{tr}}\,\sigma- (\bar{\operatorname{div}}\,\sigma)(\bar\nu)\big] +\sigma(\bar\nu,a)-\bar g(a,\bar\nu)\,\bar{\operatorname{tr}}\,\sigma\,\mathrm{d}\bar\mu\\
	&\qquad-\frac12\,\int_{B_{2\,\lambda}(0)\setminus B_{\lambda}(\lambda\,\xi)}R\,\bar g(a,\lambda^{-1}\,x-\xi)\,\mathrm{d}\bar v.
\\	&\qquad+o(1). 
	\end{aligned}
\end{equation} 
The preceding estimates may all be differentiated with respect to $\xi$. Using that $R=o(|x|^{-n})$, we see that 
$$
\int_{B_{2\,\lambda}(0)\setminus B_{\lambda}(\lambda\,\xi)}R\,\bar g(a,\lambda^{-1}\,x-\xi)\,\mathrm{d}\bar v=o(1)
$$
and 
\begin{equation} \label{r estimate intro}  
\begin{aligned} 
&\frac{d}{ds}\bigg|_{s=0}\int_{B_{2\,\lambda}(0)\setminus B_{\lambda}(\lambda\,(\xi+s\,b))}R\,\bar g(a,\lambda^{-1}\,x-(\xi+s\,b))\,\mathrm{d}\bar v\\&\qquad =-\int_{B_{2\,\lambda}(0)\setminus B_{\lambda}(\lambda\,\xi)}R\,\bar g(a,b)\,\mathrm{d}\bar v-\int_{S_{\xi,\lambda}}R\,\bar g(a,x-\lambda\,\xi)\,\bar g(\bar\nu,b)\,\mathrm{d}\bar \mu\\&\qquad =o(1)
\end{aligned} 
\end{equation} 
for every $b\in\mathbb{R}^n$ with $|b|=1$. 
The assertion of the lemma follows from Lemma \ref{mass and com convergence} and integration.
\end{proof} 

\begin{proof}[Proof of Theorem \ref{existence in all dimensions}]
	Let $\delta=1/4$. Depending on whether $m>0$ or $m<0$, Lemma \ref{g expansion} implies that $G_\lambda$ is strictly radially increasing respectively decreasing on $\{\xi\in\mathbb{R}^n:|\xi|=1/2\}$ for every $\lambda>\lambda_0$ provided that $\lambda_0$ is  large. In particular, $G_\lambda$ has a strict local minimum respectively a strict local maximum $\xi(\lambda)\in \mathbb{R}^n$ with $|\xi(\lambda)|<1/2$. According to Lemma \ref{variational to 3d}, $\Sigma(\lambda)=\Sigma_{\xi(\lambda),\lambda}$ is a constant mean curvature sphere. \\ \indent  Using  Lemma \ref{g expansion}, we find that
	\begin{align} 
	\xi(\lambda)=o(1) \label{xi lambda est}
	\end{align} 
	as $\lambda\to\infty$ and that 
	\begin{align} \label{convex or concave}
	|\bar D^2G_\lambda|\geq \frac12\,n\,(n-1)\,\omega_n\,|m|
	\end{align}
	for every   $\lambda>\lambda_0$ provided that $\lambda_0$ is large. By \eqref{convex or concave} and the implicit function theorem, the map $(\lambda_0,\infty)\to\mathbb{R}^3$, $\lambda\mapsto \xi(\lambda)$ is smooth. Arguing as in the proof of Lemma \ref{g expansion}, we find that \begin{align} \label{g'} 
	\bar DG'_\lambda=o(\lambda^{-1})\end{align}
	as $\lambda\to\infty$ uniformly for all $\xi\in\mathbb{R}^n$ with $|\xi|<1/2$.
	 Differentiating the equation $(\bar D G_\lambda)|_{\xi(\lambda)}=0$ and using \eqref{xi lambda est}, \eqref{convex or concave}, and \eqref{g'}, we find that 
	\begin{align} \label{xi'} 
	\xi'(\lambda)=[(\bar D^2G_\lambda)|_{\xi(\lambda)}]^{-1}\,(\bar DG'_\lambda)|_{\xi(\lambda)}=o(\lambda^{-1}).
	\end{align} 
	We consider the map
	$$
	\Psi: S_1(0)\times(\lambda_0,\infty)\to M\qquad\text{given by}\qquad \Psi(y,\,\lambda)=\lambda\,y+\lambda\,\xi(\lambda)+v_{\xi(\lambda),\lambda}(y)\,y.
	$$
	Note that $\Psi$ is smooth and that $\{\Psi(y,\lambda):y\in S_1(0)\}=\Sigma(\lambda)$. Moreover, using \eqref{xi}, \eqref{xi'}, \eqref{section 2 decay}, and \eqref{u estimate}, we have
	$$
	\bar g(\Psi',y)=1+O(\xi(\lambda))+O(\lambda\,\xi'(\lambda))+o(1)=1+o(1).
	$$
It follows that the spheres $\{\Sigma(\lambda):\lambda>\lambda_0\}$ form a smooth  foliation of the complement of a compact set provided that $\lambda_0>1$ is sufficiently large. \\ \indent
	Next, recall from \eqref{MC expansion} that
	$$
	H(\Sigma(\lambda))=\bar H(\Sigma(\lambda))+O(\lambda^{-1-\tau})=(n-1)\,\lambda^{-1}+O(\lambda^{-1-\tau}).
	$$ Arguing as in the derivation of \eqref{MC expansion} using \eqref{u estimate}, we see that
	$$
H(\Sigma(\lambda))'=-(n-1)\,\lambda^{-2}+O(\lambda^{-2-\tau}).
	$$
	 It follows that $\lambda\mapsto H(\Sigma(\lambda))
	$
	is strictly decreasing on $(\lambda_0,\infty)$ provided that $\lambda_0>1$ is sufficiently large. 
	\\ \indent
Recall from \eqref{stability} that $\Sigma(\lambda)$ is stable if and only if 
$$
\int_{\Sigma(\lambda)}f\,Lf\,\mathrm{d}\mu\geq 0
$$
for every $f\in C^\infty(\Sigma(\lambda))$ with
\begin{align} \label{zero mean}   
\int_{\Sigma(\lambda)}f\,\mathrm{d}\mu=0.
\end{align} Here, $L$ is the stability operator defined in \eqref{stability operator}. We decompose $f=f_0+f_1+f_2$ where $f_0\in \Lambda_0(S_{\xi(\lambda),\lambda})$, $f_1\in \Lambda_1(S_{\xi(\lambda),\lambda})$, and $f_2\perp  \Lambda_0(S_{\xi(\lambda),\lambda})\oplus\Lambda_1(S_{\xi(\lambda),\lambda})$. By scaling, we may assume that 
$$
f_1=\lambda\,\bar g(a,\bar\nu(S_{\xi(\lambda),\lambda}))
$$
for some $a\in\mathbb{R}$ with $|a|=1$ unless $f_1=0$. Note that
\begin{align*}
\lambda^{-2}\,	\int_{S_{\xi(\lambda),\lambda}} f_1^2\,\mathrm{d}\bar \mu=\omega_n\,\lambda^{n-1}
\end{align*}
and that $\bar L(S_{\xi(\lambda),\lambda})f_1=0$. Using \eqref{section 2 decay} and \eqref{u estimate}, we see that
\begin{align} \label{f1 l2 est}
\lambda^{-2}\,	\int_{\Sigma(\lambda)} f_1^2\,\mathrm{d}\mu=\omega_n\,\lambda^{n-1}+o(\lambda^{n-1})
\end{align}
and 
\begin{align} \label{Lf1} 
	Lf_1=O(\lambda^{-2-\tau}\,f_1).
\end{align}
\indent  Using \eqref{zero mean}, \eqref{section 2 decay}, and \eqref{u estimate}, we have
\begin{align} \label{s t 2 -1}
f_0=O(\lambda^{(1-n)/2-\tau}\,||f||_{L^2(\Sigma(\lambda))}).
\end{align}
Note that $\operatorname{proj}_{\Lambda_0(S_{\xi(\lambda),\lambda})^\perp}\bar L(S_{\xi(\lambda),\lambda})f_0=0$.  Using  \eqref{section 2 decay} and \eqref{u estimate},  we find that
\begin{align} \label{st 2 0}
	\operatorname{proj}_{\Lambda_0(S_{\xi(\lambda),\lambda})^\perp}Lf_0=O(\lambda^{-2-\tau}\,f_0).
\end{align}
\indent 
 Using \eqref{stability operator}, \eqref{section 2 decay}, \eqref{u estimate}, and Corollary \ref{coro st estimate}, we have
\begin{align} \label{2L2}
\int_{\Sigma(\lambda)}f_2\,Lf_2\,\mathrm{d}\mu\geq n\,\lambda^{-2}\,\int_{\Sigma(\lambda)}f_2^2\,\mathrm{d}\mu 
\end{align} 
 provided that $\lambda>1$ is sufficiently large. \\ \indent Recall that $\xi(\lambda)$ is a critical point of $G_\lambda$ and that \begin{align*} \bar D_{a,a}^2G_\lambda \begin{dcases} 
\geq& n\,(n-1)\,\omega_n\,m\,\operatorname{Id}+o(1)\quad \text{ if } m>0, \\
\leq& n\,(n-1)\,\omega_n\,m\,\operatorname{Id}+o(1)\quad\text{ if } m<0;
\end{dcases} 
\end{align*}   see Lemma \ref{g expansion}. In conjunction with Lemma \ref{variation lemma}, \eqref{section 2 decay}, and  \eqref{u estimate},  it follows that
\begin{align*} 
\lambda^{-1}\,\int_{\Sigma(\lambda)}(f_1+f_3)\,L(f_1+f_3)\,\mathrm{d}\mu \begin{dcases}  \geq&  n\,(n-1)\,\omega_n\,m+o(1) \quad  \text{ if }  m>0,\\
 \leq&  n\,(n-1)\,\omega_n\,m+o(1) \quad  \text{ if }  m<0.
\end{dcases} 
\end{align*} 
where $\lambda^{-1}\,|f_3|+|\bar \nabla f_3|+\lambda\,|\bar \nabla ^2 f_3|=O(\lambda^{-\tau})$; see \eqref{f}. Note that
$$
\lambda^{-1}\,\int_{\Sigma(\lambda)} f_3\,Lf_3\,\mathrm{d}\mu=o(1).
$$ 
Using \eqref{Lf1}, we obtain 
$$
\lambda^{-1}\,\int_{\Sigma(\lambda)} f_3\,Lf_1\,\mathrm{d}\mu=o(1).
$$
In conjunction with \eqref{f1 l2 est}, we conclude that, if $m>0$,
\begin{align} \label{m>0}
	\int_{\Sigma(\lambda)}f_1\,Lf_1\,\mathrm{d}\mu\geq n\,m\,\lambda^{-n}\,\int_{\Sigma(\lambda)} f_1^2\,\mathrm{d}\mu 
\end{align} 
and, if $m<0$,  
\begin{align} \label{m<0}
	\int_{\Sigma(\lambda)}f_1\,Lf_1\,\mathrm{d}\mu\leq n\,m\,\lambda^{-n}\,\int_{\Sigma(\lambda)} f_1^2\,\mathrm{d}\mu . 
\end{align} 
\indent Using \eqref{s t 2 -1}, we have 
\begin{align} \label{0L0} 
\int_{\Sigma(\lambda)}	f_0\,Lf_0\,\mathrm{d}\mu=o(\lambda^{-n}\,||f||^2_{L^2(\Sigma(\lambda))}).
\end{align}
Using also \eqref{st 2 0}, we have
\begin{align} \label{12L0}
	\int_{\Sigma(\lambda)} f_2\,Lf_0\,\mathrm{d}\mu=o(\lambda^{-n}\,||f||^2_{L^2(\Sigma(\lambda))})\qquad\text{and}\qquad \int_{\Sigma(\lambda)} f_1\,Lf_0\,\mathrm{d}\mu=o(\lambda^{-n} \,||f||^2_{L^2(\Sigma(\lambda))}).
\end{align} 
Finally, using \eqref{Lf1}, we have
\begin{align} \label{2L1}
	\int_{\Sigma(\lambda)} f_2\,Lf_1\,\mathrm{d}\mu=o(\lambda^{-2}\,||f_2||^2_{L^2(\Sigma(\lambda))})+o(\lambda^{-n}\,||f_1||^2_{L^2(\Sigma(\lambda))})
\end{align}
\indent Assume that $m>0$. Assembling  \eqref{2L2}, \eqref{m>0}, \eqref{0L0}, \eqref{12L0}, and \eqref{2L1}, we see that $\Sigma(\lambda)$ is stable. \\ \indent
Assume that $m<0$. We choose $f\in C^\infty(\Sigma(\lambda))$ with $f_0=f_2=0$ and $f_1\neq 0$. Using \eqref{m<0}, we see that $\Sigma(\lambda)$ is not stable. 
\\ \indent
The assertion follows.
\end{proof} 
\begin{prop} \label{local uniqueness result 1} Let $(M,g)$ be a connected complete Riemannian three-manifold that is $C^3$-asympt-otically flat with mass $m\neq 0$.
	Given $\delta>0$, there exists $r>1$ such that every stable constant mean curvature sphere $\Sigma\subset M$ that encloses $B_r$ with 
	\begin{align*} 
		\delta \,\lambda(\Sigma)<\rho(\Sigma)
	\end{align*} 
	satisfies $\Sigma=\Sigma(H)$ for some $H\in(0,H_0)$.
\end{prop} 
\begin{proof}
	 Using Lemma \ref{g expansion},  we see that, depending  on whether $m>0$ or $m<0$, $G_\lambda$ is strictly convex respectively strictly concave on $\{\xi\in\mathbb{R}^3:|\xi|<1-\delta\}$ provided that $\lambda>1$ is sufficiently large. In particular, $G_\lambda$ has at most one critical point. Using  Remark \ref{h circ estimate bar} and \eqref{H estimate}, we may now argue as in the proof of \cite[Theorem 2]{BrendleEichmair}  that every stable constant mean curvature sphere $\Sigma\subset M$ with $\delta\,\lambda(\Sigma)<\rho(\Sigma)$ and $\rho(\Sigma)>1$ sufficiently large satisfies $\Sigma=\Sigma(\lambda)$ for some $\lambda>\lambda_0$.
\end{proof}
\begin{prop} \label{local uniqueness result 2} Let $(M,g)$ be a connected complete Riemannian manifold of dimension $3\leq n\leq 7$ that is $C^3$-asymptotically flat with mass $m\neq 0$.
	Given $\delta>0$, there exists $r>1$ such that every isoperimetric hypersurface $\Sigma\subset M$ that encloses $B_r$ with 
	\begin{align*} 
		\delta \,\lambda(\Sigma)<\rho(\Sigma)
	\end{align*} 
	satisfies $\Sigma=\Sigma(H)$ for some $H\in(0,H_0)$.
\end{prop} 
\begin{proof}
	We may argue as in the proof of Proposition \ref{local uniqueness result 1}, except that we use \cite[Lemma 56]{eichmair2023schoen} instead of Remark \ref{h circ estimate bar} and \eqref{H estimate}.
\end{proof}

\section{Proof of Theorem \ref{com thm}}
Let $n\geq 3$. In this section, we assume that $g$ is a  Riemannian metric on $\mathbb{R}^n$ whose scalar curvature is integrable with $R=o(|x|^{-n})$ as $x\to\infty$ and that there is $\tau\in((n-2)/2,n-2]$ with  
\begin{align} \label{section 4 decay} 
	g=\bar g+\sigma \qquad \text{where} \qquad 
	\partial_J \sigma=O\left(|x|^{-\tau-|J|}\right)
\end{align} 
for every multi-index $J$ with $|J|\leq 3$. \\ \indent  We define the metric \begin{align} \label{tilde metric} \tilde g(x)=\frac12\, [g(x)+g(-x)]\end{align}  and  the $(0,2)$-tensors $$\tilde \sigma(x)=\frac12\,[\sigma(x)+\sigma(-x)]\qquad \text{and} \qquad \hat \sigma(x)=\frac12\,[\sigma(x)-\sigma(-x)].$$ 
We also require that
\begin{align} \label{section 4 RT}
	\partial_J \hat\sigma=O\left(|x|^{-1- \tau-|J|}\right)
\end{align} 
 for every multi-index $J$ with $|J|\leq 2$. \\ \indent	We use a  tilde to indicate that a quantity is computed with respect to the metric \eqref{tilde metric}.\\ \indent 
Let $\lambda>\lambda_0$ where $\lambda_0>1$ is large.  Recall from the proof of Theorem \ref{existence in all dimensions} that $\xi(\lambda)\in\mathbb{R}^n$ is the unique critical point of the reduced area function $G_\lambda:\{\xi\in\mathbb{R}^n:|\xi|<1/2\}\to\mathbb{R}$. We abbreviate $\Sigma(\lambda)=\Sigma_{\xi(\lambda),\lambda}$,  $S(\lambda)=S_{\xi(\lambda),\lambda}$, $u=u_{\xi(\lambda),\lambda}$, and $\Phi_\lambda=\Phi^{u_{\xi(\lambda),\lambda}}_{\xi,\lambda}$. \\ \indent 
As in, e.g., \cite[\S 2]{Huang}, given $f\in C^0(S(\lambda))$, we consider the functions $f^{\operatorname{e}},\, f^{\operatorname{o}}\in C^0(S(\lambda))$ with 
$$
f^{\operatorname{e}}=\frac12\,[f(x)+f(2\,\lambda\,\xi(\lambda)-x)]\qquad \text{and}\qquad   f^{\operatorname{o}}=\frac12\,[f(x)-f(2\,\lambda\,\xi(\lambda)-x)].
$$
Note that $f^{\operatorname{e}}$ and $f^{\operatorname{o}}$ are the symmetric respectively antisymmetric part of $f$ with respect to reflection across the center of $S(\lambda)$. We say that $f$ is even if $f^{\text{o}}=0$ and that $f$ is odd if $f^{\text{e}} = 0$. \\ \indent 
As noted in \cite[\S 3]{Huang} or \cite[Proposition 6.4]{Nerz}, the Regge-Teitelboim conditions \eqref{section 4 RT} lead to an improved estimate for $u^\text{o}$. Geometrically, this estimate means that the leaves of the foliation \eqref{cmc foliation 3} are asymptotically symmetric. We record the corresponding estimate for our setting  in the following lemma. We include the proof for completeness.

\begin{lem} There \label{u odd lemma} holds
	\begin{align} \label{u odd assumption}  
		\lambda^{-1}\,| u^{\text{o}}|+|\bar\nabla u^{\text{o}}|+\lambda\,|\bar\nabla^2u^{\text{o}}|=O(\lambda^{-\tau}\,|\xi(\lambda)|)+O(\lambda^{-1-\tau}).
	\end{align} 
\end{lem} 
\begin{proof}
	Using \eqref{u estimate} and Taylor's theorem, we compute
	\begin{equation} \label{Euclidean}
	\begin{aligned}   
	&\bar H(\Sigma(\lambda))-\bar H(S(\lambda))\\&\qquad =-\bar \Delta u-(n-1)\,\lambda^{-2}\,u+\lambda^{-1}\,u*\bar\Delta u+\lambda^{-1}\,\bar\nabla u*\bar\nabla u+\lambda^{-3}\,u*u+O(\lambda^{-1-3\,\tau}).
	\end{aligned} 
\end{equation} 
Moreover,	using \eqref{u estimate}, \eqref{mean curvature change}, and \eqref{section 4 decay}, we have
\begin{equation}\label{non Euclidean} 
\begin{aligned}  
&[H(\Sigma(\lambda))-\bar H(\Sigma(\lambda))]-[H(S(\lambda))-\bar H(S(\lambda))]\\&\qquad =\bar D^2\sigma*u+\bar D\sigma*\bar\nabla u+\sigma*\bar\nabla^2 u+O(\lambda^{-1-3\,\tau}).
\end{aligned} 
\end{equation} 
 \indent 
	Note that $H(\Sigma(\lambda))^{\text{o}}=0$. Using Lemma \ref{perturbations}, the decay assumptions \eqref{section 4 decay} and \eqref{section 4 RT}, as well as Taylor's theorem, we find that
	$$
	H(S(\lambda))^{\text{o}}=O(\lambda^{-1-\tau}\,|\xi(\lambda)|)+O(\lambda^{-2-\tau}).
	$$
Indeed, we have,	e.g., 
	$$
	\bar{\operatorname{tr}}\,\sigma(x)-\bar{\operatorname{tr}}\,\sigma(2\,\lambda\,\xi(\lambda)-x)=	\bar{\operatorname{tr}}\,\sigma(x)-\bar{\operatorname{tr}}\,\sigma(-x)+O(\lambda^{-\tau}\,|\xi(\lambda)|)=O(\lambda^{-\tau}\,|\xi(\lambda)|)+O(\lambda^{-1-\tau}).
	$$
	Likewise, using also \eqref{u estimate} and \eqref{non Euclidean}, we obtain
\begin{align*} 
&[H(\Sigma(\lambda))-\bar H(\Sigma(\lambda))]^o-[H(S(\lambda))-\bar H(S(\lambda))]^o\\&\qquad =O(\lambda^{-2-\tau}\,|u^o|)+O(\lambda^{-1-\tau}\,|\bar\nabla u^o|)+O(\lambda^{-\tau}\,|\bar\nabla^2 u^o|)+O(\lambda^{-1-2\,\tau}\,|\xi(\lambda)|)+O(\lambda^{-2-2\,\tau}).
\end{align*} 
Finally, using \eqref{Euclidean} and \eqref{u estimate}, we have
\begin{align*} 
&\bar H(\Sigma(\lambda))^o-\bar H(S(\lambda))^o\\&\qquad=-\bar \Delta u^o-(n-1)\,\lambda^{-2}\,u^o+O(\lambda^{-2-\tau}\,|u^o|)+O(\lambda^{-1-\tau}\,|\bar\nabla u^o|)+O(\lambda^{-\tau}\,|\bar\nabla^2 u^o|)+O(\lambda^{-1-3\,\tau}).
\end{align*} 
We conclude that
\begin{align*} 
&-\bar \Delta u^o-(n-1)\,\lambda^{-2}\,u^o\\&\qquad =O(\lambda^{-1-\tau}\,|\xi(\lambda)|)+O(\lambda^{-2-\tau})+O(\lambda^{-2-\tau}\,|u^o|)+O(\lambda^{-1-\tau}\,|\bar\nabla u^o|)+O(\lambda^{-\tau}\,|\bar\nabla^2 u^o|).
\end{align*} 
By Lemma \ref{spherical harmonics lemma 1}, the operator
 $$
 -\bar \Delta-(n-1)\,\lambda^{-2}: \{u\in H^2(S(\lambda)):u\perp\Lambda_1(S(\lambda))\}\to \{f\in L^2(S(\lambda)):f\perp\Lambda_1(S(\lambda))\}
 $$
 is an isomorphism.  Moreover, by Proposition \ref{cmc ls prop}, $u^o\perp \Lambda_0(S(\lambda))\oplus \Lambda_1(S(\lambda))$.  The assertion now follows from elliptic regularity.
\end{proof}
\begin{lem}
	There holds, as $\lambda\to\infty$, \label{first der comp} 
	$$
	\lambda\,\int_{\Sigma(\lambda)}H\,g(a,\nu)\,\mathrm{d}\mu=\lambda\,\int_{S(\lambda)}H\,g(a,\nu)\,\mathrm{d}\mu+o(1)+o(\lambda\,|\xi(\lambda)|).
	$$
\end{lem}
\begin{proof}
	By the first variation formula, 
	\begin{align*} 
		\int_{\Sigma(\lambda)}H\,g(a,\nu)\,\mathrm{d}\mu=	\int_{\Sigma(\lambda)}\operatorname{div}a-g(D_\nu a,\nu)\,\mathrm{d}\mu.
	\end{align*} 
Note that 
$$
\operatorname{div}\,a-\tilde{\operatorname{div}}\, a= g*\bar D\hat\sigma*a+\hat \sigma*\bar D\tilde \sigma*a+O(\lambda^{-2-2\,\tau}).
$$
In conjunction with \eqref{u estimate}, \eqref{section 4 decay}, and \eqref{section 4 RT}, we obtain, on $S(\lambda)$,
\begin{align*} 
	\operatorname{div}\,a-\tilde{\operatorname{div}}\, a&=O(\lambda^{-2-\tau})\\
(\operatorname{div}\,a-\tilde{\operatorname{div}}\, a)\circ \Phi_\lambda&=\operatorname{div}\,a-\tilde{\operatorname{div}}\, a+O(\lambda^{-2-2\,\tau})=O(\lambda^{-2-\tau}).
\end{align*} 
Moreover, using \eqref{section 4 RT} and \eqref{u estimate}, we have
\begin{align*} 
&\mathrm{d}\mu(S(\lambda))-\mathrm{d}\tilde \mu(S(\lambda))=O(\lambda^{-1-\tau}), \\ &\Phi_\lambda^*(\mathrm{d}\mu(\Sigma(\lambda))-\mathrm{d}\tilde \mu(\Sigma(\lambda)))=O(\lambda^{-1-\tau}).
\end{align*} 
Using \eqref{u estimate} and \eqref{section 4 decay}, we see that
$$
\Phi_\lambda^*(\mathrm{d}\mu(\Sigma(\lambda)))-\mathrm{d}\mu(S(\lambda))=O(\lambda^{-\tau}).
$$
It follows that
\begin{align} \label{to obtain} 
\int_{\Sigma(\lambda)}\operatorname{div}a\,\mathrm{d}\mu-\int_{S(\lambda)}\operatorname{div}a\,\mathrm{d}\mu  =\int_{\Sigma(\lambda)}\tilde {\operatorname{div}}\,a\,\mathrm{d}\tilde\mu-\int_{S(\lambda)}\tilde{\operatorname{div}}\,a\,\mathrm{d}\tilde\mu
+o(\lambda^{-1}).
\end{align} 
By an analogous argument, we obtain
	\begin{equation*} 
		\begin{aligned} 
			&\int_{\Sigma(\lambda)}g(D_\nu a,\nu)\,\mathrm{d}\mu-\int_{S(\lambda)}g(D_\nu a,\nu)\,\mathrm{d}\mu  =\int_{\Sigma(\lambda)}\tilde g(\tilde D_{\tilde\nu} a,\tilde\nu)\,\mathrm{d}\tilde\mu-\int_{S(\lambda)}\tilde g(\tilde D_{\tilde\nu} a,\tilde\nu)\,\mathrm{d}\tilde\mu
			+o(\lambda^{-1}).
		\end{aligned} 
	\end{equation*} 
\indent 	Using \eqref{u odd assumption}, \eqref{section 4 decay}, and Taylor's theorem, we find
	\begin{align*} 
		\int_{\Sigma(\lambda)}\tilde {\operatorname{div}}\,a-\tilde g(\tilde D_{\tilde\nu} a,\tilde\nu)\,\mathrm{d}\tilde\mu=\int_{\Sigma_{\xi(\lambda),\lambda}(u^{\text{e}})}\tilde {\operatorname{div}}\,a-\tilde g(\tilde D_{\tilde\nu} a,\tilde\nu)\,\mathrm{d}\tilde\mu+o(\lambda^{-1})+o(|\xi(\lambda)|).
	\end{align*} 
Using  \eqref{u estimate} and \eqref{section 4 decay}, we see that
\begin{align} \label{sff estimate} 
\tilde h(\Sigma_{s\,\xi(\lambda),\lambda}(u^e))=\lambda^{-1}\,\bar g|_{\Sigma_{s\,\xi(\lambda),\lambda}(u^e)}+O(\lambda^{-1-\tau })
\end{align} 
for all $s\in[0,1]$. Note that the normal speed of the variation $\{\Sigma_{s\,\xi(\lambda),\lambda}(u^e):s\in[0,1]\}$ with respect to $\tilde g$ is given by $\lambda\,\tilde g(\xi(\lambda),\tilde \nu)$. For every $s\in[0,1]$, there holds
\begin{align*} 
&\frac{d}{ds}\left(\tilde{\operatorname{div}}\,a-\tilde g(\tilde D_{\tilde \nu}a,\tilde \nu)\right)
\\&\qquad = \lambda\,\tilde g(\xi(\lambda),\tilde \nu)\,\tilde D_{\tilde \nu}\tilde {\operatorname{div}}\,a-\lambda\,\tilde g(\xi(\lambda),\tilde \nu)\,\tilde g(\tilde D^2_{\tilde \nu,\tilde\nu} a,\tilde\nu)+ \lambda\,\tilde g(\tilde D_{\tilde Y}a,\tilde \nu )+\lambda\,\tilde g(\tilde D_{\tilde \nu} a, \tilde Y)
\\ &\qquad \qquad+\lambda\,\tilde g(\tilde D_{\tilde X}a,\tilde \nu)+\lambda\,\tilde g(\tilde D_{\tilde \nu}a, \tilde X)
+\lambda\,\tilde D_{\xi(\lambda)^{\tilde \top}}\tilde{\operatorname{div}}\,a-\lambda\,\tilde D_{\xi(\lambda)^{\tilde \top}}\tilde g(\tilde D_{\tilde \nu}a,\tilde \nu)
\end{align*} 
where $\tilde X,\,\tilde Y\in T\Sigma_{s\,\xi(\lambda),\lambda}(u^e)$ are such that $\tilde g(\tilde X,\tilde Z)=\tilde g(\tilde D_{\tilde Z}\xi(\lambda),\tilde \nu)$ and $\tilde g(\tilde Y,\tilde Z)=\tilde h(\xi(\lambda)^{\tilde \top },\tilde Z)$ for every $\tilde Z\in T\Sigma_{s\,\xi(\lambda),\lambda}(u^e)$. Moreover, 
$$
\frac{d}{ds}\mathrm{d}\tilde \mu=\lambda\,\tilde H\,\tilde g(\xi(\lambda),\tilde\nu)\,\mathrm{d}\tilde \mu+\lambda\,\tilde{\operatorname{div}}_{\Sigma_{s\,\xi(\lambda),\lambda}(u^e)}\,\xi(\lambda)^{\tilde \top}\,\mathrm{d}\tilde \mu.
$$
In conjunction with  \eqref{section 4 decay} and \eqref{sff estimate}, we obtain using integration by parts
	\begin{align*} 
		&\int_{\Sigma_{\xi(\lambda),\lambda}(u^{\text{e}})}\tilde {\operatorname{div}}\,a-\tilde g(\tilde D_{\tilde\nu} a,\tilde\nu)\,\mathrm{d}\tilde\mu\\&\quad\qquad=\int_{\Sigma_{0,\lambda}(u^{\text{e}})}\tilde {\operatorname{div}}\,a-\tilde g(\tilde D_{\tilde\nu} a,\tilde\nu)\,\mathrm{d}\tilde\mu\\&\quad\qquad\qquad +\int_{0}^1\int_{\Sigma_{s\,\xi(\lambda),\lambda}(u^{\text{e}})}\left[\lambda\,\tilde D_{\tilde \nu}\tilde {\operatorname{div}}\,a-\lambda\,\tilde g(\tilde D^2_{\tilde \nu,\tilde\nu} a,\tilde\nu)+(n-1)\,[\tilde {\operatorname{div}}\,a-\tilde g(\tilde D_{\tilde\nu} a,\tilde\nu)]\right]\,\tilde g(\xi(\lambda),\tilde \nu)\\&\,\,\,\quad\qquad\qquad \qquad \qquad\qquad \qquad\quad+\tilde g(\tilde D_{\xi(\lambda)^{\tilde\top} }a,\tilde \nu)+\tilde g(\tilde D_{\tilde \nu}a,\xi(\lambda)^{\tilde\top}) \,\mathrm{d}\tilde\mu\,\mathrm{d}s\\&\quad\qquad\qquad+o(|\xi(\lambda)|).	
	\end{align*} 
Likewise,
\begin{align*} 
	&\int_{S(\lambda)}\tilde {\operatorname{div}}\,a-\tilde g(\tilde D_{\tilde\nu} a,\tilde\nu)\,\mathrm{d}\tilde\mu\\&\quad\qquad=\int_{S_{0,\lambda}}\tilde {\operatorname{div}}\,a-\tilde g(\tilde D_{\tilde\nu} a,\tilde\nu)\,\mathrm{d}\tilde\mu\\&\quad\qquad\qquad +\int_{0}^1\int_{S_{s\,\xi(\lambda),\lambda}}\left[\lambda\,\tilde D_{\tilde \nu}\tilde {\operatorname{div}}\,a-\lambda\,\tilde g(\tilde D^2_{\tilde \nu,\tilde\nu} a,\tilde\nu)+(n-1)\,[\tilde {\operatorname{div}}\,a-\tilde g(\tilde D_{\tilde\nu} a,\tilde\nu)]\right]\,\tilde g(\xi(\lambda),\tilde \nu)\\&\,\,\,\quad\qquad\qquad \qquad \qquad\qquad \qquad\quad+\tilde g(\tilde D_{\xi(\lambda)^{\tilde\top} }a,\tilde \nu)+\tilde g(\tilde D_{\tilde \nu}a,\xi(\lambda)^{\tilde\top}) \,\mathrm{d}\tilde\mu\,\mathrm{d}s\\&\quad\qquad\qquad+o(|\xi(\lambda)|).	
\end{align*}
Using \eqref{section 4 decay}, \eqref{u estimate}, and Taylor's theorem as in the derivation of \eqref{to obtain}, we see that, for every $s\in[0,1]$,
\begin{align*} 
&	\int_{\Sigma_{s\,\xi(\lambda),\lambda}(u^{\text{e}})}\left[\lambda\,\tilde D_{\tilde \nu}\tilde {\operatorname{div}}\,a-\lambda\,\tilde g(\tilde D^2_{\tilde \nu,\tilde\nu} a,\tilde\nu)+2\,[\tilde {\operatorname{div}}\,a-\tilde g(\tilde D_{\tilde\nu} a,\tilde\nu)]\right]\,\tilde g(\xi(\lambda),\tilde \nu)\\&\,\,\,\quad\qquad  \qquad\qquad \qquad\quad+\tilde g(\tilde D_{\xi(\lambda)^{\tilde\top} }a,\tilde \nu)+\tilde g(\tilde D_{\tilde \nu}a,\xi(\lambda)^{\tilde \top}) \,\mathrm{d}\tilde\mu
\\&\qquad =\int_{S_{s\,\xi(\lambda),\lambda}}\left[\lambda\,\tilde D_{\tilde \nu}\tilde {\operatorname{div}}\,a-\lambda\,\tilde g(\tilde D^2_{\tilde \nu,\tilde\nu} a,\tilde\nu)+2\,[\tilde {\operatorname{div}}\,a-\tilde g(\tilde D_{\tilde\nu} a,\tilde\nu)]\right]\,\tilde g(\xi(\lambda),\tilde \nu)\\&\,\,\,\quad\qquad  \qquad\qquad \qquad\quad+\tilde g(\tilde D_{\xi(\lambda)^{\tilde\top} }a,\tilde \nu)+\tilde g(\tilde D_{\tilde \nu}a,\xi(\lambda)^{\tilde\top}) \,\mathrm{d}\tilde\mu\\&\quad\qquad\qquad+o(|\xi(\lambda)|).
\end{align*} 
Consequently,
	\begin{equation*} 
		\begin{aligned} 
			&\int_{\Sigma_{\xi(\lambda),\lambda}(u^{\text{e}})}\tilde {\operatorname{div}}\,a-\tilde g(\tilde D_{\tilde\nu} a,\tilde\nu)\,\mathrm{d}\tilde\mu-\int_{S(\lambda)}\tilde {\operatorname{div}}\,a-\tilde g(\tilde D_{\tilde\nu} a,\tilde\nu)\,\mathrm{d}\tilde\mu	
			\\  &\qquad =\int_{\Sigma_{0,\lambda}(u^{\text{e}})}\tilde {\operatorname{div}}\,a-\tilde g(\tilde D_{\tilde\nu} a,\tilde\nu)\,\mathrm{d}\tilde\mu-\int_{S_{0,\lambda}}\tilde {\operatorname{div}}\,a-\tilde g(\tilde D_{\tilde\nu} a,\tilde\nu)\,\mathrm{d}\tilde\mu \\
			&\qquad\quad\qquad +o(|\xi(\lambda)|).	
		\end{aligned}
	\end{equation*}  By symmetry, we have
	\begin{equation*} 
		\begin{aligned}  
			\int_{\Sigma_{0,\lambda}(u^{\text{e}})}\tilde {\operatorname{div}}\,a-\tilde g(\tilde D_{\tilde\nu} a,\tilde\nu)\,\mathrm{d}\tilde\mu=0\qquad\text{and}\qquad  
			\int_{S_{0,\lambda}}\tilde {\operatorname{div}}\,a-\tilde g(\tilde D_{\tilde\nu} a,\tilde\nu)\,\mathrm{d}\tilde\mu=0.
		\end{aligned}
	\end{equation*}
	Assembling these estimates,	the assertion follows.
\end{proof} 
\begin{lem}
	There holds, as $\lambda\to\infty$, \label{second der comp} 
	$$
	\lambda\,	\int_{\Sigma(\lambda)}H\,g(a,\nu)\,\mathrm{d}\mu=\frac12\,\lambda\,\int_{S(\lambda)}\bar D_a\bar{\operatorname{tr}}\,\sigma-(\bar D_a\sigma)(\bar\nu,\bar\nu)\,\mathrm{d}\bar\mu+o(1)+o(\lambda\,|\xi(\lambda)|).
	$$
\end{lem}
\begin{proof}
	Let $\mathcal{M}$ be the space of $C^3$-asymptotically flat metrics on $\mathbb{R}^n$. Given $\lambda>1$, we let $$\mathcal{F}_\lambda:\{\xi\in\mathbb{R}^n:|\xi|<1/2\}\times \mathcal{M}\to\mathbb{R}$$ be given by
	$$
	\mathcal{F}_\lambda(\xi,g)=\lambda\,	\int_{S_{\xi,\lambda}}\operatorname{div}a-g(D_\nu a,\nu)\,\mathrm{d}\mu.
	$$
	Since $g$ is $C^3$-asymptotically flat, $\mathcal{F}$ is differentiable twice with respect to $\xi$. Moreover, $\mathcal{F}$ is smooth with respect to $g$. By symmetry, for every $g\in\mathcal{M}$,
	\begin{align} \label{symmetry} 
		\mathcal{F}_\lambda(0,\tilde g)=0.
	\end{align} 
By Taylor's theorem, we have	\begin{equation*} 
		\mathcal{F}_\lambda(\xi,g)=\mathcal{F}_\lambda(\xi,\bar g)+(\mathcal{D}_{\sigma}\mathcal{F}_\lambda)|_{(\xi,\bar g)}+\frac12\,(\mathcal{D}^2_{\sigma,\sigma}\mathcal{F}_\lambda)|_{(\xi,\bar g)}+\frac16\,(\mathcal{D}^3_{\sigma,\sigma,\sigma}\mathcal{F}_\lambda)|_{(\xi,\bar g)}+o(1)
	\end{equation*}
as $\lambda\to\infty$
	where $\mathcal{D}$ indicates differentiation with respect to the second variable. On the one hand, \eqref{section 4 decay}, \eqref{section 4 RT}, and Taylor's theorem imply that
	\begin{equation*}  
		\begin{aligned} 
			(\mathcal{D}^2_{\sigma,\sigma}\mathcal{F}_\lambda)|_{(\xi,\bar g)}
			=(\mathcal{D}^2_{\tilde \sigma,\tilde \sigma}\mathcal{F}_\lambda)|_{(\xi,\bar g)}+o(1)
			= (\mathcal{D}^2_{\tilde \sigma,\tilde \sigma}\mathcal{F}_\lambda)|_{(0,\bar g)}+o(1)+o(\lambda\,|\xi|).
		\end{aligned}
	\end{equation*}  
and 
	\begin{equation*}  
	\begin{aligned} 
		(\mathcal{D}^3_{\sigma,\sigma,\sigma}\mathcal{F}_\lambda)|_{(\xi,\bar g)}
		=(\mathcal{D}^3_{\tilde \sigma,\tilde \sigma,\tilde \sigma}\mathcal{F}_\lambda)|_{(\xi,\bar g)}+o(1)
		= (\mathcal{D}^3_{\tilde \sigma,\tilde \sigma,\tilde \sigma}\mathcal{F}_\lambda)|_{(0,\bar g)}+o(1)+o(\lambda\,|\xi|).
	\end{aligned}
\end{equation*}  
	On the other hand, \eqref{symmetry} implies that 
	\begin{align*}   
	(\mathcal{D}^2_{\tilde \sigma,\tilde \sigma}\mathcal{F}_\lambda)|_{(0,\bar g)}=(\mathcal{D}^3_{\tilde \sigma,\tilde \sigma,\tilde \sigma}\mathcal{F}_\lambda)|_{(0,\bar g)}=0
	\end{align*} 
while, clearly,
$$
	\mathcal{F}_\lambda(\xi,\bar g)=0.
$$
	Finally, Lemma \ref{perturbations} implies that \begin{align*}  (\mathcal{D}_\sigma\mathcal{F})|_{(\xi,\bar g)}=\frac12\,\lambda\,\int_{S_{\xi,\lambda}}\bar D_a\bar{\operatorname{tr}}\,\sigma-(\bar D_a\sigma)(\bar\nu,\bar\nu)\,\mathrm{d}\bar\mu.\end{align*} 
	Assembling these estimates and using Lemma \ref{first der comp}, the assertion follows. 
\end{proof} 
\begin{lem}
	There holds, as $\lambda\to\infty$, \label{third der comp} 
	$$
	\lambda\,\int_{\Sigma(\lambda)}H\,g(a,\nu)\,\mathrm{d}\mu=\frac{n-1}{2}\,\int_{S(\lambda)} \bar g(a,\bar\nu)\,\bar{\operatorname{tr}}\,\sigma\,\mathrm{d}\bar\mu+o(1)+o(\lambda\,|\xi(\lambda)|).
	$$
\end{lem}
\begin{proof} 
	Since $H(\Sigma(\lambda))$ is constant, we have
	\begin{align*} 
		\lambda\,\int_{\Sigma(\lambda)}H\,g(a,\nu)\,\mathrm{d}\mu =\lambda\,H(\Sigma(\lambda))\,\int_{\Sigma(\lambda)}g(a,\nu)\,\mathrm{d}\mu.
	\end{align*} 
	Arguing as in the proof of Lemma \ref{first der comp}, we obtain 
	\begin{equation*} 
		\begin{aligned} 
			\int_{\Sigma(\lambda)}g(a,\nu)\,\mathrm{d}\mu=\int_{S(\lambda)}g(a,\nu)\,\mathrm{d}\mu+o(1)+o(\lambda\,|\xi(\lambda)|).
		\end{aligned}
	\end{equation*} 
	Arguing as in the proof of Lemma \ref{second der comp} and	using Lemma \ref{perturbations}, it follows that 
	\begin{align*}  
		\int_{S(\lambda)}g(a,\nu)\,\mathrm{d}\mu=\frac12\int_{S(\lambda)} \bar g(a,\bar\nu)\,\bar{\operatorname{tr}}\,\sigma\,\mathrm{d}\bar\mu+o(1)+o(\lambda\,|\xi(\lambda)|).
	\end{align*} 
	Recall from \eqref{MC expansion} that $\lambda\,H(\Sigma(\lambda))=(n-1)+O(\lambda^{-\tau})$. Using \eqref{section 4 RT}, we find
	\begin{equation*}  
		\int_{S(\lambda)} \bar g(a,\bar\nu)\,\bar{\operatorname{tr}}\,\hat \sigma\,\mathrm{d}\bar\mu=O(\lambda^{n-2-\tau}).
	\end{equation*} 
	Similarly, using Taylor's theorem, we obtain
	\begin{equation*}  
		\int_{S(\lambda)} \bar g(a,\bar\nu)\,\bar{\operatorname{tr}}\,\tilde \sigma\,\mathrm{d}\bar\mu= \int_{S_{\lambda}(0)} \bar g(a,\bar\nu)\,\bar{\operatorname{tr} }\,\tilde\sigma\,\mathrm{d}\bar\mu+O(\lambda^{n-1-\tau}\,|\xi(\lambda)|).
	\end{equation*} 
By symmetry,
$$
\int_{S_{\lambda}(0)} \bar g(a,\bar\nu)\,\bar{\operatorname{tr}}\,\tilde \sigma\,\mathrm{d}\bar\mu=0.
$$ 
	Assembling these estimates, the assertion follows.
\end{proof}
\begin{proof}[Proof of Theorem \ref{com thm}]
	Combining Lemma \ref{second der comp} and Lemma \ref{third der comp}, we have 
	$$
	\frac12\,\lambda\, \int_{S{(\lambda)}}\bar D_a\bar{\operatorname{tr}}\,\sigma-(\bar D_a\sigma)(\bar\nu,\bar\nu)-(n-1)\,\lambda^{-1}\,\bar g(a,\bar\nu)\,\bar{\operatorname{tr}}\,\sigma\,\mathrm{d}\bar\mu=o(1)+o(\lambda\,|\xi(\lambda)|).
	$$
	Using Lemma \ref{ce int by parts rema}, we have
	\begin{align*} 
	&\lambda\, \int_{S{(\lambda)}}\bar D_a\bar{\operatorname{tr}}\,\sigma-(\bar D_a\sigma)(\bar\nu,\bar\nu)-(n-1)\,\lambda^{-1}\,\bar g(a,\bar\nu)\,\bar{\operatorname{tr}}\,\sigma\,\mathrm{d}\bar\mu\\&\qquad=	\bar g(a,\lambda\,\xi(\lambda)) \int_{S(\lambda)} (\bar{\operatorname{div}}\,\sigma)(\bar \nu)-\bar D_{\bar\nu}\bar{\operatorname{tr}}\,\sigma\,\mathrm{d}\bar\mu\\
	&\qquad \qquad\quad\,\,\,\, + \int_{S(\lambda)}\bar g(a,x)\,\big[\bar D_{\bar\nu}\bar{\operatorname{tr}}\,\sigma-(\bar{\operatorname{div}}\,\sigma)(\bar\nu)\big]+\sigma(\bar\nu,a)-\bar g(a,\bar\nu)\,\bar{\operatorname{tr}}\,\sigma\,\mathrm{d}\bar\mu
	\end{align*} 
Moreover, by \eqref{scalar curvature}, using that $R=o(|x|^{-n})$, 
\begin{align*}
	& \bar{\operatorname{div}}\left(\bar g(a,x-\lambda\,\xi(\lambda))\,[\bar D\,\bar{\operatorname{tr}}\,\sigma-(\bar{\operatorname{div}}\,\sigma)]+ [\sigma(a,\,\cdot\,)-\bar g(a,\,\cdot\,)\,\bar{\operatorname{tr}}\,\sigma]\right) =o(|x|^{1-n})+o(\lambda\,|x|^{-n}).
\end{align*} 
Using the divergence theorem, we now obtain the improved estimate 
	\begin{equation}  \label{improved estimate}
		\begin{aligned}
			&\bar g(a,\lambda\,\xi(\lambda)) \int_{S_{\lambda}(0)} (\bar{\operatorname{div}}\,\sigma)(\bar \nu)-\bar D_{\bar\nu}\bar{\operatorname{tr}}\,\sigma\,\mathrm{d}\bar\mu\\
			&\qquad\quad\,\,\,\, + \int_{S_{\lambda}(0)}\bar g(a,x)\,\big[\bar D_{\bar\nu}\operatorname{tr}\sigma-(\bar{\operatorname{div}}\,\sigma)(\bar\nu)\big]+\sigma(\bar\nu,a)-\bar g(a,\bar\nu)\,\bar{\operatorname{tr}}\,\sigma\,\mathrm{d}\bar\mu
			\\&\qquad=\,o(1)+o(\lambda\,|\xi(\lambda)|)+o\bigg(\int_{B_{\lambda}(\lambda\,\xi)\triangle B_{\lambda}(0)}|x|^{1-n}+\lambda\,|x|^{-n}\,\mathrm{d}\bar v\bigg)
			\\&\qquad=o(1)+o(\lambda\,|\xi(\lambda)|);
		\end{aligned}
	\end{equation}
cp.~\eqref{compare}.
	Since $(M,g)$ satisfies the $C^2$-Regge-Teitelboim conditions, the Hamiltonian center of mass $C=(C^1,\ldots,\,C^n)$ exists; see Proposition \ref{existence COM}.
	Using \eqref{mass} and \eqref{center of mass}, we find that
	\begin{align}\label{first bary} 
		\lambda\,\xi(\lambda)=C+o(1)+o(\lambda\,|\xi(\lambda)|).
	\end{align} 
	In particular, \begin{align} \lambda\,\xi(\lambda)=O(1) \label{barycenter est}\end{align} and in fact $\lambda\,\xi(\lambda)=C+o(1)$. Using \eqref{section 4 RT} and \eqref{u odd assumption}, we find that  
		\begin{align*} 
		|\Sigma(\lambda)|^{-1}\,\int_{\Sigma(\lambda)}x\,\mathrm{d}\mu=	|\Sigma(\lambda)|^{-1}\,\int_{\Sigma(\lambda)}x\,\mathrm{d}\tilde \mu+o(1)=|\Sigma(\lambda)|^{-1}\,\int_{\Sigma_{\xi(\lambda),\lambda}(u^e)}x\,\mathrm{d}\tilde \mu+o(1).
	\end{align*} 
	Using \eqref{section 4 decay} and \eqref{barycenter est} and Taylor's theorem, we obtain
	\begin{align*} 
|\Sigma(\lambda)|^{-1}\,\int_{\Sigma_{\xi(\lambda),\lambda}(u^e)}x\,\mathrm{d}\tilde \mu=|\Sigma(\lambda)|^{-1}\,\int_{\Sigma_{0,\lambda}(u^e)}\lambda\,\xi(\lambda)+x\,\mathrm{d}\tilde \mu+o(1).
	\end{align*} 
By symmetry, 
\begin{align*}
|\Sigma(\lambda)|^{-1}\,\int_{\Sigma_{0,\lambda}(u^e)}\lambda\,\xi(\lambda)+\lambda\,x\,\mathrm{d}\tilde \mu&=|\Sigma(\lambda)|^{-1}\,\int_{S_{0,\lambda}}\lambda\,\xi(\lambda)\,\mathrm{d}\tilde \mu\\&=|\Sigma(\lambda)|^{-1}\,\int_{S_{0,\lambda}}\lambda\,\xi(\lambda)\,\mathrm{d} \mu+o(1)\\&=\lambda\,\xi(\lambda)+o(1).
\end{align*} 
	The assertion follows from these estimates. 
\end{proof} 
	\section{Curvature estimates for stable CMC spheres} \label{curv est section}
In this section, we discuss the curvature estimates for stable constant mean curvature spheres in asymptotically flat Riemannian three-manifolds that are needed in Section \ref{uniqueness section} and in Section \ref{uniqueness section 2 }. As in \cite{chodosh2017global} and differently from, e.g.,~\cite{Ma}, we rely on a refined $L^2$-estimate of the tracefree second fundamental form obtained from combining the Christodoulou-Yau estimate \eqref{CY estimate}  with the global estimate \eqref{hi estimate r geq 0} on the Hawking mass found by G.~Huisken and T.~Ilmanen \cite{HI}. \\ \indent  
Throughout this section, we assume that $g$ is a  Riemannian metric on $\mathbb{R}^3$ whose scalar curvature is integrable and that there is  $\tau\in(1/2,1]$ with 
\begin{align} \label{section curv est decay} 
	g=\bar g+\sigma \qquad \text{where} \qquad 
	\partial_J \sigma=O\left(|x|^{-\tau-|J|}\right)
\end{align} 
for every multi-index $J$ with $|J|\leq 2$. \\ \indent 
Let $\Sigma\subset \mathbb{R}^3$ be a stable constant mean curvature sphere. Recall from \cite[p.~13]{ChristodoulouYau} that
\begin{align}\label{CY estimate}
	\frac23\,	\int_{\Sigma}|\hcirc|^2+R\,\mathrm{d}\mu\leq 16\,\pi-\int_{\Sigma}H^2\,\mathrm{d}\mu.
\end{align}
We  need the following  decay estimate.
\begin{lem}[{\cite[Lemma 5.2]{HuiskenYau}}] 
	Let $\Sigma\subset\mathbb{R}^3\setminus\{0\}$ be a closed surface and $q>2$. There is a constant $c(q)>0$ such that \label{hy integral lemma}
	\begin{align*}
		\rho(\Sigma)^{q-2}\,\int_{\Sigma}|x|^{-q}\,\mathrm{d}\bar\mu\leq c(q)\,\int_{\Sigma}\bar H^2\,\mathrm{d}\bar\mu.
	\end{align*}
\end{lem} 
Let $\{\Sigma_i\}_{i=1}^\infty$ be a sequence of stable constant mean curvature spheres $\Sigma_i\subset \mathbb{R}^3$ with
\begin{align*}
	\lim_{i\to\infty}\rho(\Sigma_i)=\infty\qquad\text{and}\qquad \rho(\Sigma_i)=O(\lambda(\Sigma_i)).
\end{align*} 
We recall the following two results.
\begin{lem}[{\cite[Lemma 2.3]{Ma}}]
	There holds, as $i\to\infty$, \label{large blowdown}
	\begin{align*}
		\int_{\Sigma_i}\bar H^2\,\mathrm{d}\bar\mu=16\,\pi+o(1).
	\end{align*} 
\end{lem}
\begin{proof}
	This follows from \eqref{CY estimate} and Lemma \ref{hy integral lemma}; see \cite{Ma} for details.
\end{proof}

\begin{lem}[cp.~{\cite[Proposition D.1]{chodosh2017global}}]
	Assume that $R\geq 0$. Then, as $i\to\infty$, \label{CE estimate}
	\begin{align} \label{hi estimate r geq 0}
		16\,\pi-\int_{\Sigma_i} H^2\,\mathrm{d}\mu \leq O(\lambda(\Sigma_i)^{-1}).
	\end{align}
Alternatively,	assume that $R=O(|x|^{-5/2-\tau})$ as $x\to\infty$. Then 
	\begin{align} \label{hi estimate}
		16\,\pi-\int_{\Sigma_i} H^2\,\mathrm{d}\mu \leq O(\rho(\Sigma_i)^{-1/2-\tau})+O(\lambda(\Sigma_i)^{-1}).
	\end{align}
\end{lem}
\begin{proof}
	The estimate \eqref{hi estimate r geq 0} is proved in {\cite[Proposition D.1]{chodosh2017global}}. \\ \indent To obtain \eqref{hi estimate}, we adapt the argument in {\cite[Proposition D.1]{chodosh2017global}} as follows. \\ \indent Let $\Sigma_i'\subset \mathbb{R}^3$ be the minimizing hull of $\Sigma_i$;  see \cite[p.~371]{HI}. Using \cite[(1.15)]{HI}, we have
	\begin{align} \label{mh comparison}
		16\,\pi-\int_{\Sigma_i} H^2\,\mathrm{d}\mu\leq 16\,\pi-\int_{\Sigma'_i}H^2\,\mathrm{d}\mu.
	\end{align}
	Moreover, there holds
	\begin{align} \label{lambda mh comp}
		\lambda(\Sigma_i)=(1+o(1))\,\lambda(\Sigma_i'),
	\end{align}
	see \cite[(26)]{chodosh2017global}, while, clearly, 
	\begin{align} \label{rho mh comp} 
		\rho(\Sigma'_i)\geq \rho(\Sigma_i).
	\end{align} Let $u_i\in C^{1,1}(\mathbb{R}^3)$ be the proper weak solution of inverse mean curvature flow with initial data $\Sigma_i'$ in the sense of \cite[p.~365]{HI}. Using the growth formula \cite[(5.22)]{HI} for the right-hand side of  \eqref{mh comparison}, the co-area formula, and arguing as in \cite[Appendix H]{CESH}, we find that 
	\begin{align} \label{willmore mh estimate}
		16\,\pi-\int_{\Sigma'_i}H^2\,\mathrm{d}\mu\leq O(\lambda(\Sigma_i')^{-1})+O\bigg(\int_{\mathbb{R}^3\setminus B_{\rho(\Sigma_i')}(0)}|Du_i|\,R\,\mathrm{d}\mu\bigg).
	\end{align}
	By \eqref{CY estimate} and Lemma \ref{hy integral lemma}, we have $H(\Sigma_i)=O(\lambda(\Sigma_i)^{-1})$. In conjunction with \cite[(3.1)]{HI}, \cite[(1.15)]{HI}, and \eqref{section curv est decay}, we find that \begin{align} \label{gradient mh est}
		|Du_i|=O(\lambda(\Sigma_i)^{-1})+O(|x|^{-1})
	\end{align}
	on $\mathbb{R}^3$, uniformly as $i\to\infty$. Assembling \eqref{mh comparison}-\eqref{gradient mh est} and using that $R=(|x|^{-5/2-\tau})$ as $x\to\infty$, the assertion follows.
\end{proof} 

\begin{lem} \label{basic h estimate}
	Assume that $R\geq 0$. Then, as $i\to\infty$,
	\begin{align*} 
		H(\Sigma_i)=2\,\lambda(\Sigma_i)^{-1}+O(\lambda(\Sigma_i)^{-2}).
	\end{align*} 
Alternatively, assume that	$R=O(|x|^{-5/2-\tau})$ as $x\to\infty$. Then
	\begin{align*} 
		H(\Sigma_i)=2\,\lambda(\Sigma_i)^{-1}+O(\rho(\Sigma_i)^{-1/2-\tau}\,\lambda(\Sigma_i)^{-1})+O(\lambda(\Sigma_i)^{-2}).
	\end{align*} 
\end{lem} 
\begin{proof}
	If $R\geq 0$, we find, using \eqref{CY estimate} and Lemma \ref{CE estimate}, that
	$$
	16\,\pi-\int_{\Sigma_i}H^2\,\mathrm{d}\mu=O(\lambda(\Sigma_i)^{-1}).
	$$
	Alternatively, if $R=O(|x|^{-5/2-\tau})$, we use \eqref{CY estimate},  Lemma \ref{hy integral lemma}, and Lemma \ref{CE estimate} to find that
	$$
	16\,\pi-\int_{\Sigma_i}H^2\,\mathrm{d}\mu=O(\rho(\Sigma_i)^{-1/2-\tau})+O(\lambda(\Sigma_i)^{-1}).
	$$
	Moreover, we have
	$$
\int_{\Sigma_i}H^2\,\mathrm{d}\mu=4\,\pi\,H(\Sigma_i)^2\,\lambda(\Sigma_i)^2
	$$
	for all $i$.  Note that $H(\Sigma_i)>0$ by the maximum principle. The assertion follows.
\end{proof} 
\begin{lem}
	Assume that $R\geq 0$. Then, as $i\to\infty$, 
	\begin{align*}
		\int_{\Sigma_i} |\hcirc|^2\,\mathrm{d}\mu=O(\lambda(\Sigma_i)^{-1}).
	\end{align*}
Alternatively, assume that $R=O(|x|^{-5/2-\tau})$ as $x\to\infty$. Then \label{hcirc l2 estimate}  
	\begin{align*}
		\int_{\Sigma_i} |\hcirc|^2\,\mathrm{d}\mu=O(\rho(\Sigma_i)^{-1/2-\tau})+O(\lambda(\Sigma_i)^{-1}).
	\end{align*}
\end{lem}
\begin{proof}
	This follows from \eqref{CY estimate}, Lemma \ref{CE estimate}, and Lemma \ref{hy integral lemma}.
\end{proof} 
\begin{prop}
	Assume that $R\geq 0$. Then, as $i\to\infty$, \label{h circ estimate}
	\begin{equation}\label{hcirc r pos}\begin{aligned} 
			\hcirc&=O(|x|^{-1-\tau})+O(|x|^{-1}\,\lambda(\Sigma_i)^{-1/2}),\\
h&=	O(\lambda(\Sigma_i)^{-1})+O(|x|^{-1-\tau})+O(|x|^{-1}\,\lambda(\Sigma_i)^{-1/2}).
\end{aligned}\end{equation}
Alternatively, assume that	$R=O(|x|^{-5/2-\tau})$ as $x\to\infty$. Then \label{prop CE}
	\begin{equation}\begin{aligned} \label{hcirc}
			\hcirc&=O(|x|^{-1}\,\rho(\Sigma_i)^{-1/4-\tau/2})+O(|x|^{-1}\,\lambda(\Sigma_i)^{-1/2}),\\
			h&=O(\lambda(\Sigma_i)^{-1})+\,O(|x|^{-1}\,\rho(\Sigma_i)^{-1/4-\tau/2})+O(|x|^{-1}\,\lambda(\Sigma_i)^{-1/2}).
	\end{aligned}\end{equation}
	In either case, 	\begin{align} \label{H estimate}
		\bar H&=H+O(|x|^{-1-\tau}).
	\end{align}
\end{prop}
\begin{proof} As shown in \cite[Theorem 2.7]{Ma}, the Simons' identity and the Sobolev inequality imply that
	\begin{equation*}
		\begin{aligned}
			|\hcirc|^2=\,&O(|x|^{-2-2\,\tau})+O\bigg(|x|^{-2}\,\int_{\Sigma_i}|\hcirc|^2\,\mathrm{d}\mu \bigg)
		\end{aligned}
	\end{equation*}
	as $i\to\infty$.  In conjunction with Lemma \ref{hcirc l2 estimate}, we obtain \eqref{hcirc r pos} and \eqref{hcirc}. Now, \eqref{H estimate} follows from the estimate
	$
	\bar H=H+O(|x|^{-\tau}\,|h|)+O(|x|^{-1-\tau}).
	$
\end{proof} 
\begin{rema}
In \cite[Theorem 2.7]{Ma}, 	\eqref{section curv est decay} is required to hold for every multi-index $J$ with $|J|\leq 3$. The arguments in \cite[Proposition 3.3]{Metzger} show that it is sufficient to require \eqref{section curv est decay} for every multi-index $J$ with $|J|\leq 2$.
\end{rema}
\begin{rema}
	It follows from Proposition \ref{prop CE} that, if $R\geq 0$, then, as $i\to\infty,$  \label{h circ estimate bar}
	\begin{equation*}\begin{aligned} 
			&\hbarcirc=O(|x|^{-1-\tau})+O(|x|^{-1}\,\lambda(\Sigma_i)^{-1/2}),\\
			&\bar h=O(\lambda(\Sigma_i)^{-1})+O(|x|^{-1-\tau})+O(|x|^{-1}\,\lambda(\Sigma_i)^{-1/2})
	\end{aligned}\end{equation*}
while,	if $R=O(|x|^{-5/2-\tau})$ as $x\to\infty$, 
	\begin{equation*}\begin{aligned}
			\hbarcirc&=O(|x|^{-1}\,\rho(\Sigma_i)^{-1/4-\tau/2})+O(|x|^{-1}\,\lambda(\Sigma_i)^{-1/2}),\\
\bar h&=O(\lambda(\Sigma_i)^{-1})+O(|x|^{-1}\,\rho(\Sigma_i)^{-1/4-\tau/2})+O(|x|^{-1}\,\lambda(\Sigma_i)^{-1/2}).			
	\end{aligned}\end{equation*} 
\end{rema}

\section{Proof of Theorem \ref{uniqueness thm}} \label{sec:uniqueness}
In this section, we assume that $g$ is a Riemannian metric on $\mathbb{R}^3$ whose scalar curvature is integrable with $R=o(|x|^{-3})$ as $x\to\infty$ and that there is $\tau\in(1/2,1]$ with  
\begin{align*} 
g=\bar g+\sigma \qquad \text{where} \qquad 
\partial_J \sigma=O\left(|x|^{-\tau-|J|}\right)
\end{align*} 
for every multi-index $J$ with $|J|\leq 3$.
\\ \indent  Let $\{\Sigma_i\}_{i=1}^\infty$ be a sequence of stable constant mean \label{uniqueness section} curvature spheres $\Sigma_i\subset \mathbb{R}^3$ that enclose $B_1(0)$ with 
\begin{align} \label{divergence}
\lim_{i\to\infty}\rho(\Sigma_i)=\infty
\end{align} 
and
\begin{align} \label{slow divergence}
\rho(\Sigma_i)=o(\lambda(\Sigma_i))
\end{align} 
as $i\to\infty$. By \cite[Lemma 1.1]{Simon},  Lemma \ref{large blowdown}, and \eqref{slow divergence},
\begin{align}
\label{x vs lambda} \sup_{x\in\Sigma_i}|x|=O(\lambda(\Sigma_i)).
\end{align}
\indent Let $x_i\in \Sigma_i\cap S_{\rho(\Sigma_i)}(0).$ Passing to a subsequence, we may assume that there is $\xi\in\mathbb{R}^3$ with $|\xi|=1$ and
\begin{align} \label{xi limit} 
\lim_{i\to\infty} |x_i|^{-1}\,x_i=-\xi. 
\end{align}

\begin{lem}[{cp.~\cite[Corollary 4.7]{chodosh2017global}}]
 \label{distance comp} 
The surfaces $\tfrac12\,H(\Sigma_i)\,\Sigma_i$  converge to $S_1(\xi)$ in $C^1$ in $\mathbb{R}^3$.
\end{lem} 
\begin{proof}
	We may assume that $\xi=e_3$. Let $a_i\in\mathbb{R}^3$ with $|a_i|=1$, $a_i\perp x_i$, and $a_i\perp e_3$. Let $R_i\in SO(3)$ be the unique rotation with $R(a_i)=a_i$ and $R(x_i)=-|x_i|\,e_3$.  
 By \eqref{xi limit}, $\lim_{i\to\infty}R_i=\operatorname{Id}$. \\ \indent   Let $\gamma_i>0$ be the largest radius such that there is a smooth function  $u_i:\{y\in\mathbb{R}^2:|y|\leq \gamma_i\}\to\mathbb{R}$ with 
\begin{equation} 
\label{gradient estimate}
\begin{aligned}
\circ\qquad &|(\bar\nabla u_i)|_y|\leq1,\qquad\qquad\qquad\qquad\qquad\qquad\qquad\qquad\qquad\qquad \\[-3.5pt]
\circ\qquad &(y,\rho(\Sigma_i)+u_i(y))\in R_i(\Sigma_i)
\end{aligned} 
\end{equation}
for all $y\in\mathbb{R}^2$ with $|y|\leq \gamma_i$.
  Clearly, $\gamma_i>0$ and $(\nabla u_i)|_0=0$. It follows that 
\begin{align} \label{phi vs y} 
4\,|(y,\rho(\Sigma_i)+u_i(y))|\geq|y|+\rho(\Sigma_i)
\end{align} 
 and
\begin{align} \label{first h} 
|(\bar\nabla^2 u_i)|_y|\leq 8\,|\bar h(R_i(\Sigma_i))((y,\rho(\Sigma_i)+u_i(y)))|
\end{align} 
for every $y\in \mathbb{R}^2$ with $|y|\leq \gamma_i$. Moreover, Lemma \ref{basic h estimate}, \eqref{H estimate}, 
the improved curvature estimates  in Remark \ref{h circ estimate bar}, and \eqref{x vs lambda} imply that
\begin{equation} \label{improved grad estimate}  
\begin{aligned} 
\bar h (R_i(\Sigma_i))=&\,\frac12\,H(\Sigma_i)\,\bar g|_{R_i(\Sigma_i)}+O(|x|^{-1-\tau})+O(|x|^{-1}\,H(\Sigma_i)^{1/2})\\=\,&\frac12\,H(\Sigma_i)\,\bar g|_{R_i(\Sigma_i)}+O(|x|^{-3/2}).
\end{aligned} 
\end{equation} 
Combining \eqref{phi vs y}-\eqref{improved grad estimate},   we have
$$
|(\bar\nabla^2 u_i)|_y|\leq 4\,H(\Sigma_i)+O((|y|+\rho(\Sigma_i))^{-3/2}).
$$
Integrating, 
\begin{align} \label{grad estimate} 
|(\bar\nabla u_i)|_y|\leq 4\,H(\Sigma_i)\,|y|+O(\rho(\Sigma_i)^{-1/2}).
\end{align} 
It follows that $\tfrac12\,H(\Sigma_i)\,\gamma_i\geq \tfrac1{16}$ for all $i$ sufficiently large.  \eqref{grad estimate} also shows that, given $\varepsilon>0$, there is $\delta>0$ such that
$$
|\nu(R_i(\Sigma_i))-e_3|\leq \varepsilon\qquad\text{on}\qquad \big\{(y,\rho(\Sigma_i)+u_i(y)):y\in\mathbb{R}^2\text{ with }\tfrac12\, H_i\,|y|\leq \delta \big\}.
$$
Finally, Lemma \ref{large blowdown}, \cite[Theorem 3.1]{Simon}, and \eqref{improved grad estimate}  imply that $\tfrac12\,H(\Sigma_i)\,R_i(\Sigma_i)$ converges to $S_1(\tilde \xi)$ in $C^2$  locally  in $\mathbb{R}^3\setminus\{0\}$ where $\tilde \xi\in\mathbb{R}^3$; see also \cite[Lemma 3.1]{QingTian} and \cite[Proposition 2.2]{chodosh2017global}. 
The preceding argument shows that $\tilde \xi=\xi$ and that the convergence is in $C^1$ in $\mathbb{R}^3$.
\end{proof} 

\begin{proof}[Proof of Theorem \ref{uniqueness thm}]
Suppose, for a contradiction, that the conclusion of Theorem \ref{uniqueness thm} fails. Using Proposition \ref{local uniqueness result 1}, it follows that there is a sequence $\{\Sigma_i\}_{i=1}^\infty$ of stable constant mean curvature spheres $\Sigma_i\subset \mathbb{R}^3$ enclosing $B_1(0)$ that satisfies \eqref{divergence} and \eqref{slow divergence}.  \\ \indent  
	Let $a\in\mathbb{R}^3$ with $|a|=1$.
As in \cite[(5.13)]{HuiskenYau}, our starting point is the identity
\begin{align} \label{4 0}
\int_{\Sigma_i} H\,g(a,\nu)\,\mathrm{d}\mu=H(\Sigma_i)\,\int_{\Sigma_i}g(a,\nu)\,\mathrm{d}\mu.
\end{align} 
 On the one hand, Lemma \ref{q perturbations} implies that
$$
g(a,\nu)\,\mathrm{d}\mu=\left[\bar g(a,\bar\nu)+\frac12\,\bar g(a,\bar\nu)\,\bar{\operatorname{tr}}\,\sigma+O(|x|^{-2\,\tau})\right]\,\mathrm{d}\bar\mu
$$
uniformly on $\Sigma_i$  as $i\to\infty$. Moreover, by the divergence theorem, 
$$
\int_{\Sigma_i}\,\bar g(a,\bar\nu)\,\mathrm{d}\bar \mu=0.
$$
In conjunction with Lemma \ref{basic h estimate},   \eqref{H estimate}, and Lemma \ref{hy integral lemma}, we obtain
\begin{align} \label{4 1}
H(\Sigma_i)	\int_{\Sigma_i} g(a,\nu)\,{d}\mu=\frac12\,\int_{\Sigma_i}\bar H\,\bar g(a,\bar\nu)\,\bar{\operatorname{tr}}\,\sigma\,\mathrm{d}\bar \mu+o(1).
\end{align} 
On the other hand, by the first variation formula, we have
\begin{align} \label{4 1 5}
\int_{\Sigma_i} H\,g(a,\nu)\,\mathrm{d}\mu=	\int_{\Sigma_i} \operatorname{div}a-g(D_\nu a,\nu)\,\mathrm{d}\mu.
\end{align} 
By Lemma \ref{q perturbations}, 
$$
[\operatorname{div}a-g(D_\nu a,\nu)]\,\mathrm{d}\mu=\frac12 [\bar D_a\bar{\operatorname{tr}}\,\sigma -(\bar D_a\sigma)(\bar\nu,\bar\nu)+O(|x|^{-1-2\,\tau})]\,\mathrm{d}\bar\mu
$$
uniformly on $\Sigma_i$  as $i\to\infty$. In conjunction with Lemma \ref{hy integral lemma}, we find
\begin{align} \label{4 2}
\int_{\Sigma_i} \operatorname{div}a-g(D_\nu a,\nu)\,\mathrm{d}\mu=\frac12 \int_{\Sigma_i} \bar D_a\bar{\operatorname{tr}}\,\sigma -(\bar D_a\sigma)(\bar\nu,\bar\nu)\,\mathrm{d}\bar\mu+o(1).
\end{align} 
 Assembling \eqref{4 0}-\eqref{4 2} and using Lemma \ref{ce int by parts new} , we conclude that
\begin{equation*}
\begin{aligned}  
0&=\int_{\Sigma}\big[\bar D_{\bar\nu}\bar{\operatorname{tr}}\,\sigma- (\bar{\operatorname{div}}\,\sigma)(\bar\nu)\big]\,\bar g(a,\bar\nu)\,\mathrm{d}\bar\mu +\frac12\,\int_{\Sigma_i}\bar H\, [\sigma(a,\bar\nu)-\bar{\operatorname{tr}}\,\sigma\,\bar g(a,\bar\nu)]\,\mathrm{d}\bar\mu
\\&\qquad+O\bigg(\int_{\Sigma_i}|\hbarcirc|\,|\sigma|\,\mathrm{d}\bar \mu\bigg)+o(1).
\end{aligned}
\end{equation*}
Using also \eqref{H estimate} and Lemma \ref{hy integral lemma}, we obtain
\begin{equation}  \label{4 4}
	\begin{aligned} 
0&=\int_{\Sigma_i}[\bar D_{\bar\nu}\bar{\operatorname{tr}}\,\sigma-(\bar{\operatorname{div}}\,\sigma)(\bar\nu)]\,\bar g(a,\bar\nu)\,\mathrm{d}\bar\mu+\frac12\,H(\Sigma_i)\,\int_{\Sigma_i} \sigma(a,\bar\nu)-\bar{\operatorname{tr}}\,\sigma\,\bar g(a,\bar\nu)\,\mathrm{d}\bar\mu\\&\qquad+O\bigg(\int_{\Sigma_i}|\hbarcirc|\,|\sigma|\,\mathrm{d}\bar \mu\bigg)+o(1).
\end{aligned} 
\end{equation} 
\indent Note that 
\begin{align} \label{4 5}  
\int_{\Sigma_i}|\hbarcirc|\,|\sigma|\,\mathrm{d}\bar \mu=O\bigg(\int_{\Sigma_i}|\hbarcirc|\,|x|^{-\tau}\,\mathrm{d}\bar\mu\bigg)=o(1)
\end{align} 
by Remark \ref{h circ estimate bar} and Lemma \ref{hy integral lemma}.
Let $z_i\in \Sigma_i$ with $\bar\nu(z_i)=-|x_i|^{-1}\,x_i$ and  $$
\xi_i=\frac12\,H(\Sigma_i)\,z_i-\bar\nu(z_i).
$$ 
Note that $z_i$ is unique for sufficiently large $i$ by Lemma \ref{distance comp}. Moreover,  $2\,|z_i|\geq H(\Sigma_i)$  and
\begin{align}
|\xi_i|=1+o(1). \label{xi} 
\end{align}
We define the map $E_i:\Sigma_i\to\mathbb{R}^3$ by
$$
E_i=\bar\nu-\frac{1}{2}\,H(\Sigma_i)\,x+\xi_i.
$$
Using   Remark \ref{h circ estimate bar} and \eqref{H estimate},  we have
\begin{align} \label{der est} 
\bar\nabla E_i=O(|x|^{-3/2}).
\end{align} 
Integrating and using Lemma \ref{distance comp}, this gives
$
E_i=O(|x|^{-1/2}).
$
 In conjunction with \eqref{4 4}, \eqref{4 5}, and Lemma \ref{hy integral lemma},  we obtain 
\begin{equation*}  
\begin{aligned} 
0=\,&\int_{\Sigma_i}[\bar D_{\bar\nu}\bar{\operatorname{tr}}\,\sigma-(\bar{\operatorname{div}}\,\sigma)(\bar\nu)]\,\bar g\big(a,\tfrac12\,H(\Sigma_i)\,x-\xi_i\big)\,\mathrm{d}\bar\mu+\frac12\,H(\Sigma_i)\,\int_{\Sigma_i} \sigma(a,\bar\nu)-\bar g(a,\bar\nu)\,\bar{\operatorname{tr}}\,\sigma\,\mathrm{d}\bar\mu\\&\qquad+o(1).
\end{aligned} 
\end{equation*} 
 \indent As in the proof of Lemma \ref{g expansion}, \eqref{scalar curvature} gives
 \begin{align*}
& \bar{\operatorname{div}}\left([\bar D\,\bar{\operatorname{tr}}\,\sigma-(\bar{\operatorname{div}}\,\sigma)]\,\bar g\big(a,\tfrac12\,H(\Sigma_i)\,x-\,\xi_i\big)\,+\frac12\,H(\Sigma_i)\, [\sigma(a,\,\cdot\,)-\bar g(a,\,\cdot\,)\,\bar{\operatorname{tr}}\,\sigma]\right)\\&\qquad =-R\,\bar g\big(a,\tfrac12\,H(\Sigma_i)\,x-\,\xi_i\big)+O(|x|^{-2-2\,\tau}).
 \end{align*} 
  Using the divergence theorem and that $R$ is integrable,  we find
\begin{align*} 
0=	  &\,\bar g(a,\xi_i)\int_{S_{H(\Sigma_i)^{-1}}(0)} (\bar{\operatorname{div}}\,\sigma)(\bar\nu)-\bar D_{\bar\nu}\bar{\operatorname{tr}}\,\sigma\,\mathrm{d}\bar\mu\\
&\qquad+\frac12\,H(\Sigma_i)\, \int_{S_{H(\Sigma_i)^{-1}}(0)}\bar g(a,x)\,\big[\bar D_{\bar\nu}\bar{\operatorname{tr}}\,\sigma- (\bar{\operatorname{div}}\,\sigma)(\bar\nu)\big] +\sigma(\bar\nu,a)-\bar g(a,\bar\nu)\,\bar{\operatorname{tr}}\,\sigma\,\mathrm{d}\bar\mu\\
&\qquad +o(1);
\end{align*}
see Figure \ref{slow divergence figure}.
	\begin{figure}\centering
	\includegraphics[width=0.5\linewidth]{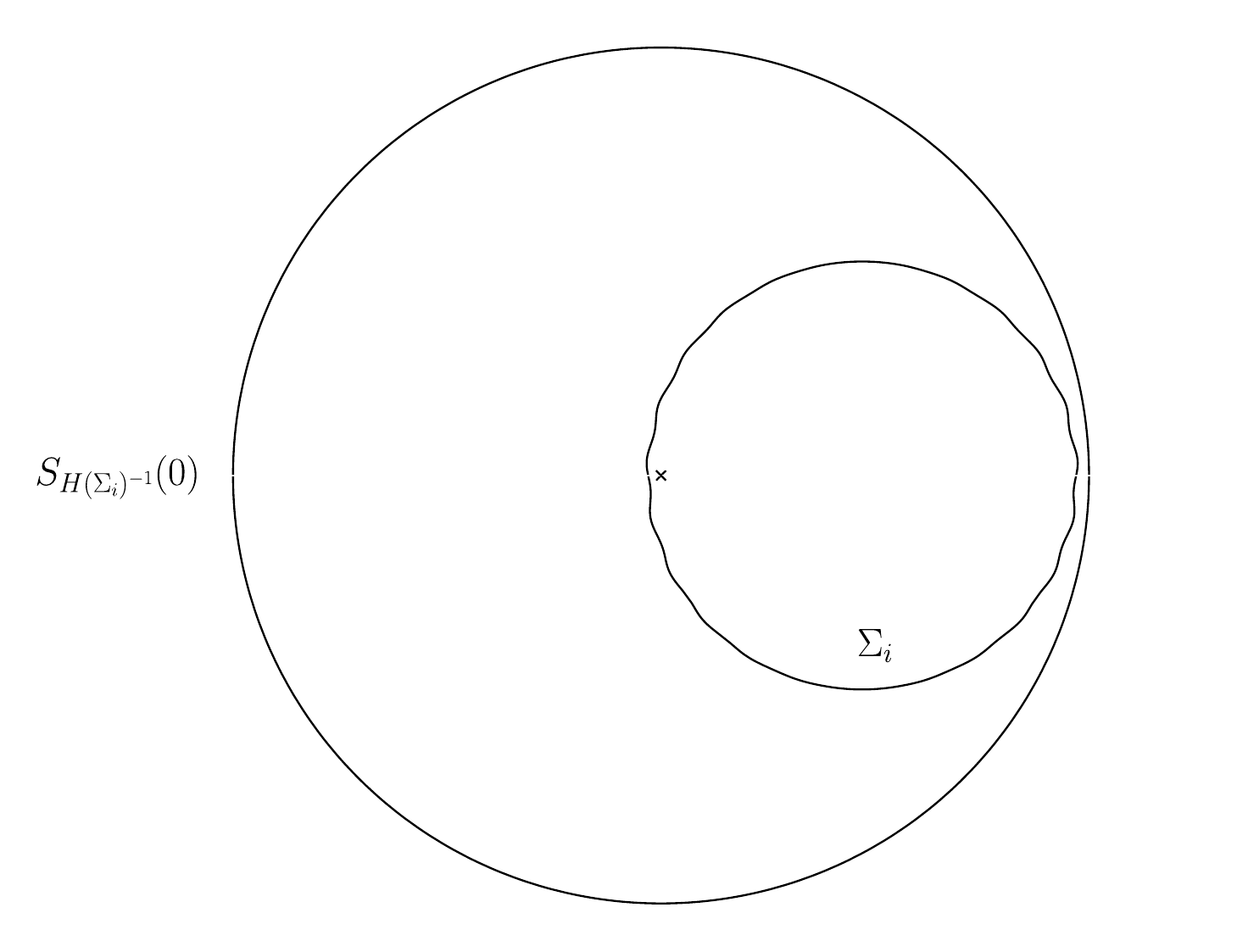}
	\caption{An illustration of the proof of Theorem \ref{uniqueness thm}. Using the divergence theorem, the flux integrals in \eqref{4 0} on $\Sigma_i$ can be computed on the sphere $S_{H(\Sigma_i)^{-1}}(0)$. The cross marks the origin in the asymptotically flat chart.} 
	\label{slow divergence figure}
\end{figure}
In conjunction with Lemma \ref{mass and com convergence}, we conclude that
$$
0=16\,\pi\,m\,\bar g(a,\xi_i)+o(1)
$$
 for every $a\in\mathbb{R}^3$ with $|a|=1$. This is incompatible with  \eqref{xi}.
\end{proof}

\section{Proof of Theorem \ref{uniqueness thm 2}} \label{sec:uniqueness 2}
In this section, we assume that $g$ is a Riemannian metric on $\mathbb{R}^3$  and that there is $\tau\in(1/2,1]$ with  
\begin{align*} 
g=\bar g+\sigma \qquad \text{where} \qquad 
\partial_J \sigma=O\left(|x|^{-\tau-|J|}\right)
\end{align*} 
for every multi-index $J$ with $|J|\leq 3$ and
$$
 R=O(|x|^{-5/2-\tau}).
$$
\indent  Let $\{\Sigma_i\}_{i=1}^\infty$ be a sequence of stable constant mean \label{uniqueness section 2 } curvature spheres $\Sigma_i\subset \mathbb{R}^3$ each enclosing $B_1(0)$ such that
\begin{align} \label{slow divergence 2} 
\lim_{i\to\infty}\rho(\Sigma_i)=\infty
\qquad\text{and}\qquad 
\rho(\Sigma_i)=o(\lambda(\Sigma_i))
\end{align} 
as $i\to\infty$. \\
\indent Let $x_i\in \Sigma_i\cap S_{\rho(\Sigma_i)}(0).$ Passing to a subsequence, we may assume that there is $\xi\in\mathbb{R}^3$ with $|\xi|=1$ and
$
\lim_{i\to\infty} |x_i|^{-1}\,x_i=-\xi. 
$

\begin{lem}
	Suppose \label{distance comp 2} that, for some $s> 1$,
\begin{align} \label{s}  
	\lambda(\Sigma_i)=O(\rho({\Sigma_i})^s).
\end{align} 
	The surfaces $\tfrac12\,H(\Sigma_i)\,\Sigma_i$  converge to $S_1(\xi)$ in $C^1$ in $\mathbb{R}^3$.
\end{lem} 
\begin{proof}
		Using 	\eqref{x vs lambda} and \eqref{s}, we have
		$$
\sup_{x\in\Sigma_i}		|x|=O(\rho(\Sigma_i)^{s}).
		$$
		In conjunction with Lemma \ref{basic h estimate}, \eqref{H estimate}, and Remark \ref{h circ estimate bar}, we obtain
		\begin{align} \label{improved grad estimate 2}  
		\bar h (\Sigma_i)=\frac12\,H(\Sigma_i)\,\bar g|_{\Sigma_i}+O(|x|^{-1-(1/4+\tau/2)\,1/s})+O(|x|^{-1}\,H(\Sigma_i)^{1/2}).
		\end{align}
		We may now argue as in the proof of Lemma \ref{distance comp} using \eqref{improved grad estimate 2} instead of \eqref{improved grad estimate}.
\end{proof}

\begin{proof}[Proof of Theorem \ref{uniqueness thm 2}]
	The proof is similar to that of Theorem \ref{uniqueness thm}. We only point out the necessary modifications. \\ \indent 
	Let $s>1$ be as in \eqref{s u 2}.
	If the conclusion of Theorem \ref{uniqueness thm 2} fails,
	 there is a sequence $\{\Sigma_i\}_{i=1}^\infty$ of stable constant mean curvature spheres $\Sigma_i\subset \mathbb{R}^3$ each enclosing $B_1(0)$ and such that \eqref{slow divergence 2} and
$
	\lambda(\Sigma_i)=O(\rho(\Sigma_i)^s)
$
hold. 	In particular, 
	\begin{align} \label{1 s}
	\lambda(\Sigma_i)=o(\rho(\Sigma_i)^{s'})
	\end{align}
	for every $s'>s$. Let $\tau'\in(1/2,\tau)$ be such that
	\begin{align} \label{2 s}
	s<1+\frac{3}{4}\,\frac{2\,\tau-1}{1-\tau'}+\frac{1}{2}\,\frac{\tau'-\tau}{1-\tau'}.
	\end{align}
	It follows  from \eqref{1 s} and \eqref{2 s} that 
	\begin{equation} \label{3 s} 
	\begin{aligned} 
	O(\lambda(\Sigma_i)^{1-\tau'}\,\rho(\Sigma_i)^{-1/4-\tau/2-(\tau-\tau')})=&\,o(1),\\
	 O(\lambda(\Sigma_i)^{1-\tau'}\,\rho(\Sigma_i)^{-1/4-\tau/2-(\tau-\tau')/2})=&\,o(1).
	\end{aligned}
	\end{equation} 
	  \indent 
Compared to the proof of Theorem \ref{uniqueness thm}, to obtain \eqref{4 5}, we now use Remark \ref{h circ estimate bar} to  estimate
	\begin{align*}
|\hbarcirc|\,|\sigma|=\,&O(|x|^{-1-\tau}\,\rho(\Sigma_i)^{-1/4-\tau/2})+O(|x|^{-1-\tau}\,\lambda(\Sigma_i)^{-1/2})\\=\,&O(\lambda(\Sigma_i)^{1-\tau'}\,|x|^{-2-(\tau-\tau')}\,\rho(\Sigma_i)^{-1/4-\tau/2})+O(|x|^{-3/2-\tau}).
	\end{align*}
	Using  Lemma \ref{hy integral lemma} and \eqref{3 s}, we  obtain
	\begin{align*}  
	\int_{\Sigma_i}|\hbarcirc|\,|\sigma|\,\mathrm{d}\bar\mu=O(\lambda(\Sigma_i)^{1-\tau'}\,\rho(\Sigma_i)^{-1/4-\tau/2-(\tau-\tau')})+O(\rho(\Sigma_i)^{1/2-\tau})=o(1).
	\end{align*} 
	\indent
	Instead of \eqref{der est}, we now apply  Remark \ref{h circ estimate bar} and \eqref{H estimate} to estimate \begin{align*} 
	\bar\nabla E_i=O(|x|^{-1}\,\rho(\Sigma_i)^{-1/4-\tau/2})+O(|x|^{-1}\,\lambda(\Sigma)^{-1/2}).
	\end{align*} 
	Lemma \ref{distance comp 2} and integration give \begin{align*} E_i=\,&O(|x|^{(\tau-\tau ')/2}\,\rho(\Sigma_i)^{-1/4-\tau/2})+O(|x|^{(\tau-\tau ')/2}\,\lambda(\Sigma_i)^{-1/2})\\=\,&O(\lambda(\Sigma_i)^{1-\tau'}\,|x|^{-1+\tau-(\tau-\tau')/2}\,\rho(\Sigma_i)^{-1/4-\tau/2})+O(|x|^{-1/2+(\tau-\tau')/2}).
	\end{align*}
	By \eqref{3 s} and Lemma \ref{hy integral lemma},   
\begin{equation*}  
\begin{aligned} 
0=\,&\int_{\Sigma_i}[\bar D_{\bar\nu}\bar{\operatorname{tr}}\,\sigma-(\bar{\operatorname{div}}\,\sigma)(\bar\nu)]\,\bar g\big(a,\tfrac12\,H(\Sigma_i)\,x-\xi_i\big)\,\mathrm{d}\bar\mu+\frac12\,H(\Sigma_i)\,\int_{\Sigma_i} \sigma(a,\bar\nu)-\bar g(a,\bar\nu)\,\bar{\operatorname{tr}}\,\sigma\,\mathrm{d}\bar\mu\\&\qquad+o(1).
\end{aligned} 
\end{equation*} 
	The argument concludes exactly as in the proof of Theorem \ref{uniqueness thm}.
\end{proof} 
	
\begin{appendices}
		\section{Asymptotically flat manifolds}
	\label{af manifolds appendix}
	Let $n\geq 3$ and $k,\,\ell\geq 2$ be integers. \\ \indent  A metric $g$ on $\{x\in\mathbb{R}^n:|x|>1/2\}$ is called $C^k$-asymptotically flat if its scalar curvature is integrable with $R=o(|x|^{-n})$ as $x\to\infty$ and if there are $\tau\in((n-2)/2,n-2]$ and a symmetric $(0,2)$-tensor $\sigma$ such that 
	\begin{equation} \label{asymptotically flat}
		\begin{aligned}  
	&g=\bar g+\sigma\qquad\text{where}  \qquad 
	\partial_J \sigma=O\left(|x|^{-\tau-|J|}\right) 
\end{aligned} 	
\end{equation} 
 for every multi-index $J$ with $|J|\leq k$. If, more strongly, for some $m\in\mathbb{R}$,
	 \begin{equation} \label{Schwarzschild metric}  
	\begin{aligned}  
	g=&\,\bigg(1+\frac{m}{2}\,|x|^{2-n}\bigg)^{\frac{4}{n-2}}\,\bar g+\sigma\qquad\text{where}\qquad 
	\partial_J \sigma=O\left(|x|^{-2-|J|}\right)
	\end{aligned} 
	\end{equation} 
 for every multi-index $J$ with $|J|\leq k$, then $g$ is called $C^k$-asymptotic to Schwarzschild with mass $m$.\\ 
	 \indent 
	A  connected complete Riemannian manifold $(M,g)$ of dimension $n$ is said to be $C^k$-asymptotically flat if there is a nonempty compact subset of $M$ whose complement is diffeomorphic to $\{x\in\mathbb{R}^n:|x|>1/2\}$ and  the pull-back of $g$ by this diffeomorphism is $C^k$-asymptotically flat. Such a diffeomorphism is called an asymptotically flat chart.  We sometimes say that
	$(M,g)$ is $C^k$-asymptotically flat of rate $\tau>(n-2)/2$ to stress the specific decay rate. If there is an asymptotically flat chart such that the pull-back metric  takes the form \eqref{Schwarzschild metric}, $(M,g)$ is called $C^k$-asymptotic to Schwarzschild with mass $m$. R.~Bartnik has shown that the integral in \eqref{mass} of a $C^2$-asymptotically flat manifold converges and that the limit does not depend on the choice of asymptotically flat chart; see \cite[Theorem 4.2]{Bartnik}. \\ \indent  We usually fix an asymptotically flat reference chart and write $B_r$, $r\geq1$, to denote the open, bounded domain in $(M,g)$ whose boundary corresponds to $S_r(0)=\{x\in\mathbb{R}^n:|x|=r\}$ in this chart. \\ \indent 
Let $(M,g)$ be a connected complete Riemannian manifold $(M,g)$ of dimensions $n$ that is $C^k$-asymptotically flat.	Following \cite[p.~292]{RT}, we say that $(M,g)$  satisfies the $C^\ell$-Regge-Teitelboim conditions   if there is $\hat \tau\in((n-2)/2,n-2]$ with
\begin{equation} \label{RT} 
	\begin{aligned}
	\partial_J\hat g&=O\left(|x|^{-1-\hat\tau-|J|}\right)  \text{ and } 
	 \\ \hat R&=O(|x|^{-3-2\,\hat\tau })
\end{aligned} 
\end{equation}
 for every multi-index $J$ with $|J|\leq \ell$. Here, 
	\begin{equation*}
		\begin{aligned} 
	\hat g(x)=\frac12\,(g(x)-g(-x))\qquad\text{and}\qquad \hat R(x)=\frac12\,(R(x)-R(-x)).	\end{aligned}
\end{equation*} 
 Following \cite[Definition 6.2]{Nerz}, if there is $\hat \tau\in((n-2)/2,n-2]$ with 
\begin{equation} \label{RT Nerz}
	\begin{aligned}  
	\partial_J\hat g&=O\left(|x|^{-1/2-\hat\tau-|J|}\right)   \text{ and } 
	\\\hat R&=O(|x|^{-5/2-2\,\hat\tau})
	\end{aligned} 
\end{equation}  
 for every multi-index $J$ with $|J|\leq \ell$,	we say that $(M,g)$ satisfies the weak $C^\ell$-Regge-Teitelboim conditions. \\ \indent
In the case where $n=3$,	L.-H.~Huang has verified that the limit in \eqref{center of mass} exists if $(M,g)$ is $C^3$-asymptotically flat and satisfies the $C^2$-Regge-Teitelboim conditions; see \cite[Theorem 2.2]{Huang2}. In Proposition \ref{existence COM}, we have included a proof that  \eqref{center of mass} also converges in the case where $n\geq 3$ and $(M,g)$ is $C^2$-asymptotically flat and satisfies the $C^2$-Regge-Teitelboim conditions.    The center of mass depends on the choice of asymptotically flat chart.  Recall from \cite[Corollary 3.2]{Bartnik} that any two asymptotically flat charts are asymptotically equal up to a Euclidean isometry. If different asymptotically flat charts are used to compute the center of mass, then the results are related by the same isometry as the  charts; see \cite[Theorem 3.1]{Huang2}.
	\section{First and second variation of area and volume}
	\label{CMC appendix}
	In this section, we collect results on the first and second variation of area and volume from  \cite[Appendix H]{CCE} that are used repeatedly in this paper.\\ \indent  Let $(M,g)$ be a  Riemannian manifold of dimension $n\geq 3$ and $\Sigma\subset M$ a  closed two-sided hypersurface with unit normal $\nu$. We also assume that $\Sigma\cap\partial M=\emptyset$.   Let $\varepsilon>0$ and $U\in C^{\infty}(\Sigma\times(-\varepsilon,\varepsilon))$ with $U(\,\cdot\,,\,0)=0$. 
	For $\varepsilon>0$ sufficiently small the map, 
		$$
	\Psi:\Sigma\times(-\varepsilon,\varepsilon)\to M\qquad\text{given by}\qquad \Psi(x,s)=\operatorname{exp}_x(U(x,s)\,\nu(x))
	$$
	is an embedding. Let $\Sigma_s=\Psi(\Sigma,s)$. Note that $\{\Sigma_s:s\in(-\varepsilon,\varepsilon)\}$ is a smooth variation of $\Sigma=\Sigma_0$. In fact, every smooth family of hypersurfaces near $\Sigma$ can be parametrized in this way. 	We denote the initial speed and initial acceleration of $\Sigma=\Sigma_0$ by
	$$
	f(x)=\dot{U}(x,\,0)\qquad \text{  and } \qquad a(x)=\ddot{U}(x,\,0).
	$$
Let
	\begin{align}
	Lf=-\Delta f-(|h|^2+\operatorname{Ric}(\nu,\nu))\,f,
	\label{stability operator}
	\end{align}
 where $\Delta$ is the nonpositive Laplace operator on $\Sigma$ with respect to the induced metric, $h$ the second fundamental form of $\Sigma$, and $\operatorname{Ric}$  the  Ricci curvature of $(M,g)$. Recall that
	\begin{align}
	Lu=\frac{d}{ds}\bigg|_{s=0}(H(\Sigma_s)\circ\Psi(\,\cdot\,,s)). \label{mean curvature change}
	\end{align}
	In Lemma \ref{variation lemma} below,
	$$
	\operatorname{vol}(\Sigma_s)=\begin{dcases}&\int_{\Sigma\times[0,s]}\Psi^*(\mathrm{d}v)\qquad\text{if }s\geq 0 \\&\int_{\Sigma\times[s,0]}\Psi^*(\mathrm{d}v)\qquad\text{if }s\leq 0 
	\end{dcases}
	$$
	denotes the volume enclosed by $\Sigma_s$ relative to $\Sigma_0=\Sigma$.
	\begin{lem} There holds \label{variation lemma}
		\begin{align*}
		\frac{d}{ds}\bigg|_{s=0}|\Sigma_s|=\int_{\Sigma} H\,f\,\mathrm{d}\mu\qquad \text{and} \qquad 	\frac{d}{ds}\bigg|_{s=0}\operatorname{vol}(\Sigma_s)=\int_{\Sigma} f\,\mathrm{d}\mu.
		\end{align*}
		Moreover,
		\begin{align*}
		\frac{d^2}{ds^2}\bigg|_{s=0}|\Sigma_s|=\int_{\Sigma}f\,Lf+H^2\,f^2+H\,a\,\mathrm{d}\mu \qquad \text{and} \qquad \frac{d^2}{ds^2}\bigg|_{s=0}\operatorname{vol}(\Sigma_s)=\int_{\Sigma} H\,f^2+a\,\mathrm{d}\mu.
		\end{align*}
	\end{lem}
The variation $\{\Sigma_s:s\in(-\varepsilon,\varepsilon)\}$ is called volume-preserving if
$
\operatorname{vol}(\Sigma_s)=\operatorname{vol}(\Sigma)
$
for all $s\in(-\varepsilon,\varepsilon)$.
	 Hypersurfaces that are critical  for the area functional among all volume-preserving variations have constant mean curvature. They are called constant mean curvature hypersurfaces.   \\ \indent A constant mean curvature hypersurface $\Sigma$ is stable if it passes the second derivative test for  area  among all volume-preserving variations. 	It can be shown that $\Sigma$ is stable if and only if 
	\begin{align} \label{stability} 
	\int_{\Sigma} f\,Lf\,\mathrm{d}\mu\geq0
	\end{align} 
	for every $f\in C^{\infty}(\Sigma)$ with
	$$
	\int_{\Sigma}f\,\mathrm{d}\mu=0.
	$$

	\section{The Laplace operator on the unit sphere}
In this section, we collect some standard results about the Laplace operator on the unit sphere. \\ \indent
Let $n\geq 3$ be an integer. Recall that $S_1(0)=\{x\in\mathbb{R}^n:|x|=1\}$.
	\begin{lem}\label{spherical harmonics lemma 1}
		The eigenvalues of the operator 
		$$-\bar\Delta:H^{2}(S_1(0))\to L^2(S_1(0))$$
		are given by
		$$
		\{\ell\,(\ell+n-2):\ell=0\,,1\,,2,\dots\}.
		$$	
	\end{lem}
	We denote the eigenspace corresponding to the eigenvalue $\ell\,(\ell+n-2)$ by $$\Lambda_\ell(S_1(0))=\{f\in C^{\infty}(S_1(0)):-\bar\Delta f=\ell\,(\ell+n-2)f\}.$$ 
	Recall that these eigenspaces are finite-dimensional and that
	$$
	L^2(S_1(0))=\bigoplus_{\ell=0}^\infty \Lambda_\ell(S_1(0)). 
	$$
Moreover, 
		$\Lambda_0(S_1(0))=\operatorname{span}\{1\}$ and
		$\Lambda_1(S_1(0))=\operatorname{span}\{x^1,\ldots,x^n\}$.
\begin{coro}
	Let $f\in C^\infty(S_1(0))$ with $f\perp \{1,\,x^1,\ldots, x^n\}$. There holds \label{coro st estimate}
	$$
	-\int_{S_1(0)}f\,\bar \Delta f+(n-1)\,f^2\,\mathrm{d}\bar\mu\geq (n+1)\,\int_{S_1(0)}f^2\,\mathrm{d}\bar\mu.
	$$
\end{coro}

	\section{Mass and center of mass}
In this section, we collect some observations on the flux integrals that define the mass \eqref{mass} and the center of mass \eqref{center of mass}. \\ \indent 
Let $n\geq 3$. We assume that $g$ is a  Riemannian metric on $\mathbb{R}^n$ with integrable scalar curvature and that there is $\tau\in((n-2)/2,n-2]$ with
	\begin{equation} \label{mcm decay} 
	\begin{aligned}  
		&g=\bar g+\sigma\qquad\text{where}  \qquad 
		\partial_J \sigma=O\left(|x|^{-\tau-|J|}\right)
	\end{aligned} 	
\end{equation}  
for every multi-index $J$ with $|J|\leq 2$.
	\begin{lem}
		There holds, as $\lambda\to\infty$, \label{mass and com convergence}
		\begin{align} \label{first}  
		\frac{1}{2\,n\,(n-1)\,\omega_n}\int_{S_\lambda(0)}
		(\bar{\operatorname{div}}\,\sigma)(\bar\nu)-\bar D_{\bar\nu}\bar{\operatorname{tr}}\,\sigma\,\mathrm{d}\bar\mu=m+o(1)
		\end{align} 
				and
		\begin{align}  \label{second} 
		\int_{S_{\lambda}(0)}\big[\bar D_{\bar\nu}\bar{\operatorname{tr}}\,\sigma- (\bar{\operatorname{div}}\,\sigma)(\bar\nu)\big]\,\bar\nu+\lambda^{-1}\,\sum_{i=1}^n\sigma(\bar\nu,e_i)\,e_i-\lambda^{-1}\,\bar{\operatorname{tr}}\,\sigma\, \bar\nu\,\mathrm{d}\bar\mu=o(1).
		\end{align}

	\end{lem}
	\begin{proof} 
		Note that the limit of the left-hand side of \eqref{first} as $\lambda\to\infty$ equals the mass \eqref{mass}. The existence of the limit  follows from the divergence theorem, \eqref{scalar curvature}, and the assumption that $R$ is integrable; see, e.g.,~\cite{Bartnik}.
		  \\ \indent
		To verify \eqref{second}, we define the map 
		$
		F:(1,\infty)\to\mathbb{R}^n
		$
	 by
		\begin{align} \label{F} 
		F(\lambda)=\int_{S_{\lambda}(0)}\big[\bar D_{\bar\nu}\bar{\operatorname{tr}}\,\sigma- (\bar{\operatorname{div}}\,\sigma)(\bar\nu)\big]\,x+\sum_{i=1}^n\sigma(\bar\nu,e_i)\,e_i-\bar{\operatorname{tr}}\,\sigma\,\bar\nu\,\mathrm{d}\bar\mu.
		\end{align} 
		Using \eqref{scalar curvature} and \eqref{mcm decay}, we have
\begin{align*} 
	\bar{\operatorname{div}}\left[\big(\bar D\,\bar{\operatorname{tr}}\,\sigma- \bar{\operatorname{div}}\,\sigma\big)\,x^\ell+\sigma(\,\cdot\,,e_\ell)-\bar{\operatorname{tr}}\,\sigma\,\bar g(\,\cdot\,,e_\ell)\right]=R\,x^\ell+O(|x|^{-1-2\,\tau}) 
\end{align*} 
as $x\to\infty$ for every $\ell=1,\ldots,n$. Using the divergence theorem and that $R$ is integrable, we conclude that
		$$(\lambda^{-1}\,F)'=-\lambda^{-2}\,F+h$$
		as $\lambda\to\infty$ where  $h$ is integrable. It follows that $\lambda^{-1}\,F=o(1)$. The assertion follows.
	\end{proof}	
For Proposition \ref{existence COM} below, recall that
	\begin{equation*}
	\begin{aligned} 
		\hat g(x)=\frac12\,(g(x)-g(-x))\qquad\text{and}\qquad \hat R(x)=\frac12\,(R(x)-R(-x)).	\end{aligned}
\end{equation*} 
\begin{prop} \label{existence COM}  
Suppose that there is $\hat \tau\in((n-2)/2,n-2]$ with
\begin{align} \label{mcm rt}  
\partial_J\hat g=O\left(|x|^{-1-\hat\tau-|J|}\right)\qquad\text{and}\qquad \hat R=O(|x|^{-3-2\,\hat\tau })
\end{align} 
for every multi-index $J$ with $|J|\leq 2$. Then  the quantities
\begin{align*} 
		\int_{S_{\lambda}(0)}\big[\bar D_{\bar\nu}\bar{\operatorname{tr}}\,\sigma- (\bar{\operatorname{div}}\,\sigma)(\bar\nu)\big]\,x+\sum_{i=1}^n\sigma(\bar\nu,e_i)\,e_i-\bar{\operatorname{tr}}\,\sigma\,\bar\nu\,\mathrm{d}\bar\mu
\end{align*}
converge as $\lambda\to\infty$.
\end{prop}
	\begin{proof}
Let
	$
	F:(1,\infty)\to\mathbb{R}^n
	$
be as in \eqref{F}. Using \eqref{scalar curvature}, we have, for every $\ell=1,\ldots,n$,
	\begin{align*} 
	&\bar{\operatorname{div}}\left[\big(\bar D\,\bar{\operatorname{tr}}\,\sigma- \bar{\operatorname{div}}\,\sigma\big)\,x^\ell+\sigma(\,\cdot\,,e_\ell)-\bar{\operatorname{tr}}\,\sigma\,\bar g(\,\cdot\,,e_\ell)\right]
	\\&\qquad =R\,x^\ell+\sigma*\bar D^2\sigma\,x^\ell+\bar D\sigma*\bar D\sigma\,x^\ell+\sigma*\sigma*\bar D^2\sigma\,x^\ell+\sigma*\bar D\sigma*\bar D\sigma\,x^\ell+O(|x|^{-1-4\,\tau}). 
	\end{align*} 
 Using \eqref{mcm decay} and \eqref{mcm rt}, we conclude that
	\begin{align*} F'(\lambda)&=\int_{S_\lambda(0)}R\,x^\ell+\sigma*\bar D^2\sigma\,x^\ell+\bar D\sigma*\bar D\sigma\,x^\ell+\sigma*\sigma*\bar D^2\sigma\,x^\ell+\sigma*\bar D\sigma*\bar D\sigma\,x^\ell+O(|x|^{-1-4\,\tau})\,\mathrm{d}\bar {\mu}\\&=O(\lambda^{n-3-2\,\hat\tau})+O(\lambda^{n-3-\tau-\hat\tau})+O(\lambda^{n-2-4\,\tau}),\end{align*} 
 In particular,   $F'$ is integrable. The assertion follows.
\end{proof}	
The following integration by parts formula has been proven in \cite{chodoshfar} in the case where $n=3$. The case where $n>3$ requires only formal modifications.
\begin{lem} [{\cite[p.~168-169]{chodoshfar}}]
	\label{ce int by parts rema} Let $\xi\in\mathbb{R}^n$ and $\lambda>0$. There holds 
	\begin{align*}& \int_{S_{\xi,\lambda}}\bar D_a\bar{\operatorname{tr}}\,\sigma-(\bar D_a\sigma)(\bar\nu,\bar\nu)-(n-1)\,\lambda^{-1}\,\bar g(a,\bar\nu)\,\bar{\operatorname{tr}}\,\sigma\,\mathrm{d}\bar\mu
	\\&\qquad =\lambda^{-1}\,\int_{S_{\xi,\lambda}}\bar g(a,x-\lambda\,\xi)\,\big[\bar D_{\bar\nu}\bar{\operatorname{tr}}\,\sigma- (\bar{\operatorname{div}}\,\sigma)(\bar\nu)\big] +\sigma(\bar\nu,a)-\bar g(a,\bar\nu)\,\bar{\operatorname{tr}}\,\sigma\,\mathrm{d}\bar\mu
	\end{align*} 
	for every  $a\in\mathbb{R}^n$.
\end{lem}
In the following lemma, we adapt Lemma \ref{ce int by parts rema} to a general closed hypersurface $\Sigma\subset\mathbb{R}^n$.
\begin{lem} Let  $\Sigma\subset\mathbb{R}^n$ be a closed hypersurface.  \label{ce int by parts new} There holds 
	\begin{align*}& \int_{\Sigma}\bar D_a\bar{\operatorname{tr}}\,\sigma-(\bar D_a\sigma)(\bar\nu,\bar\nu)-\bar H\,\bar{\operatorname{tr}}\,\sigma\,\bar g(a,\bar\nu)\,\mathrm{d}\bar\mu
	\\&\qquad =\int_{\Sigma}\bar g(a,\bar\nu)\,\big[\bar D_{\bar\nu}\bar{\operatorname{tr}}\,\sigma- (\bar{\operatorname{div}}\,\sigma)(\bar\nu)\big] +\frac1{n-1}\,\bar H\,\left[\sigma(a,\bar\nu)-\bar g(a,\bar\nu)\,\bar{\operatorname{tr}}\,\sigma\right]\,\mathrm{d}\bar\mu \\&\qquad\quad+\int_{\Sigma}\bar g({a^\top}\lrcorner \hbarcirc,(\bar\nu\lrcorner\sigma)|_\Sigma)-\bar g(a,\bar\nu)\,\bar g(\hbarcirc,\sigma|_\Sigma)\,\mathrm{d}\bar\mu
	\end{align*}
		for every  $a\in\mathbb{R}^n$.
\end{lem} 
\begin{proof}
	Let $\{f_1,\ldots, f_{n-1}\}$ be a local Euclidean orthonormal frame of $T\Sigma$. We have
	\begin{equation*}
	\begin{aligned}  
	\bar D_a\bar{\operatorname{tr}}\,\sigma-(\bar D_a\sigma)(\bar\nu,\bar\nu)=\,&\bar g(a,\bar\nu)\,\big[\bar D_{\bar\nu}\bar{\operatorname{tr}}\,\sigma- (\bar{\operatorname{div}}\,\sigma)(\bar\nu)\big]\\&\,+\bar D_{a^\top}\bar{\operatorname{tr}}\,\sigma-(\bar D_{a^\top}\sigma)(\bar\nu,\bar\nu)+\bar g(a,\bar\nu)\,\sum_{\alpha=1}^{n-1}(\bar D_{f_\alpha}\sigma)(f_\alpha,\bar\nu).
	\end{aligned} 
	\end{equation*}
	Note that
	\begin{equation*}\begin{aligned}
&	\bar D_{a^\top}\bar{\operatorname{tr}}\,\sigma-(\bar D_{a^\top}\sigma)(\bar\nu,\bar\nu)+\bar g(a,\bar\nu)\,\sum_{\alpha=1}^{n-1}(\bar D_{f_\alpha}\sigma)(f_\alpha,\bar\nu)-\bar H\,\bar{\operatorname{tr}}\,\sigma\,\bar g(a,\bar\nu) \\&\qquad =
	\bar{\operatorname{div}}_\Sigma\left((\bar{\operatorname{tr}}\,\sigma)\,a^\top-\sigma(\bar\nu,\bar\nu)\,a^\top+\bar g(a,\bar\nu)\,\sum_{\alpha=1}^{n-1}\sigma(f_\alpha,\bar\nu)\,f_\alpha\right )
	\\
	&\qquad\quad +\frac1{n-1}\,\bar H\,\big[\sigma(a,\bar\nu)-\bar g(a,\bar\nu)\,\bar{\operatorname{tr}}\,\sigma\big]+\bar g({a^\top}\lrcorner \hbarcirc,(\bar\nu\lrcorner\sigma)|_\Sigma)-\bar g(a,\bar\nu)\,\bar g(\hbarcirc,\sigma|_\Sigma).
	\end{aligned}\end{equation*} 
	The assertion follows from these identities and the first variation formula.
\end{proof} 

\section{Some geometric expansions}
	We collect several geometric expansions needed in this paper.
	\\ \indent 
	Let $n\geq 3$. We assume that $g$ is a  Riemannian metric on $\mathbb{R}^n$  and that there is $\tau\in((n-2)/2,n-2]$ with
	\begin{align*} 
		g=\bar g+\sigma \qquad \text{where} \qquad 
		\partial_J \sigma=O\left(|x|^{-\tau-|J|}\right)
	\end{align*} 
for every multi-index $J$ with $|J|\leq 2$.\\ \indent 
	  Recall that $S_{\xi,\lambda}=\{x\in\mathbb{R}^n:|x-\lambda\,\xi|=\lambda\}$ where $\xi\in\mathbb{R}^n$ and $\lambda>1$ and that $e_1,\ldots,e_n$ is the standard basis of $\mathbb{R}^n$.
\begin{lem} Let $\delta\in(0,1/2)$ and $a\in\mathbb{R}^n$ with $|a|=1$.  There holds
	\begin{equation*} 
	\begin{aligned}
	\operatorname{div}a=&\,\frac12\,\bar D_a\bar{\operatorname{tr}}\,\sigma +O(\lambda^{-1-2\,\tau}),\\
	g(D_\nu a,\nu)=&\,\frac12\,(\bar D_a \sigma)(\bar\nu,\bar\nu)+O(\lambda^{-1-2\,\tau}), \\
	g(a,\nu)=&\,\bar g(a,\bar\nu)+ \frac12\,\bar g(a,\bar\nu)\,\sigma(\bar\nu,\bar\nu)+O(\lambda^{-2\,\tau})
	\end{aligned}
	\end{equation*} 
	on $S_{\xi,\lambda}$  uniformly for all $\xi\in\mathbb{R}^n$ with $|\xi|<1-\delta$. Moreover, 
		\label{perturbations}
	\begin{align*}
	\nu-\bar\nu=&\,\frac12\,\sigma(\bar\nu,\bar\nu)\,\bar\nu-\sum_{k=1}^{n}\sigma(\bar\nu,e_k)\,e_k+O(\lambda^{-2\,\tau}),\\
		\mathrm{d}\mu=&\,\left[1+\frac{1}{2}[\bar{\operatorname{tr}}\,\sigma-\sigma(\bar\nu,\bar\nu)]+O(\lambda^{-2\,\tau})\right]\,\mathrm{d}\bar\mu,
	\\	H-\bar H=&\,\lambda^{-1}\,\left[2\,\sigma(\bar\nu,\bar\nu)-\bar{\operatorname{tr}}\,\sigma\right]+\frac12 \,\left[\bar D_{\bar\nu}\bar{\operatorname{tr}}\,\sigma+(\bar D_{\bar\nu}\sigma) (\bar\nu,\bar\nu)-2\,(\bar{\operatorname{div}}\,\sigma)(\bar\nu)\right]\\&\qquad
+\lambda^{-1}\,\sigma*\sigma*\bar\nu*\bar\nu*\bar\nu*\bar\nu+\sigma*\bar D\sigma*\bar\nu
\\&\qquad	+O(\lambda^{-1-3\,\tau}),
\\h-\bar h=&\,O(\lambda^{-1-\tau}).
	\end{align*}
These expressions may be differentiated once with respect to $\xi$. 
\end{lem}

\begin{lem}
	
	Let $\{\Sigma_i\}^\infty_{i=1}$ be a sequence \label{q perturbations} of surfaces $\Sigma_i\subset \mathbb{R}^n$ with $\lim_{i\to\infty}\operatorname{dist}(\Sigma_i,0)=\infty$. Let $a\in\mathbb{R}^n$ with $|a|=1$.  There holds 
	\begin{equation*} 
	\begin{aligned}
	\operatorname{div}a=&\,\frac12\,\bar D_a\bar{\operatorname{tr}}\,\sigma +O(|x|^{-1-2\,\tau}),\\
	g(D_\nu a,\nu)=&\,\frac12\,(\bar D_a \sigma)(\bar\nu,\bar\nu)+O(|x|^{-1-2\,\tau}),\\
		g(a,\nu)=&\,\bar g(a,\bar\nu)+ \frac12\,\bar g(a,\bar\nu)\,\sigma(\bar\nu,\bar\nu)+O(|x|^{-2\,\tau})
	\end{aligned}
	\end{equation*} 
	as $i\to\infty$. Moreover, 
	\begin{equation*}
	\begin{aligned}
	\nu-\bar\nu=&\,\frac12\,\sigma(\bar\nu,\bar\nu)\,\bar\nu-\sum_{k=1}^{n}\sigma(\bar\nu,e_k)\,e_k+O(|x|^{-2\,\tau}), \\
	\mathrm{d}\mu=&\,\left[1+\frac{1}{2}[\bar{\operatorname{tr}}\,\sigma-\sigma(\bar\nu,\bar\nu)]+O(|x|^{-2\,\tau})\right]\,\mathrm{d}\bar\mu.
	\end{aligned}
	\end{equation*}
\end{lem}
\begin{lem}
	There holds
	\begin{equation} \label{scalar curvature} 
		\begin{aligned} 
	R=\,&\bar{\operatorname{div}}\,\bar{\operatorname{div}}\,\sigma-\bar\Delta\bar{\operatorname{tr}}\,\sigma+\sigma*\bar D^2\sigma+\bar D\sigma*\bar D\sigma
+\sigma*\sigma*\bar D^2\sigma+\sigma*\bar D\sigma*\bar D\sigma+O(|x|^{-2-4\,\tau}).
	\end{aligned} 
	\end{equation}
\end{lem}

\section{Stable CMC spheres in asymptotically Schwarzschild manifolds} \label{Schwarzschild appendix} Let $(M,g)$ be a connected complete Riemannian manifold of dimension $n= 3$ that is $C^4$-asymptotic to Schwarzschild \eqref{Schwarzschild metric} with mass $m>0$. In their pioneering paper \cite{HuiskenYau},  G.~Huisken and S.-T.~Yau have shown  that there is a distinguished family \begin{align} \label{cmc foliation} \{\Sigma(H):H\in(0,H_0)\},
\end{align}
 $H_0>0$,  of stable constant mean curvature spheres $\Sigma(H)\subset M$ with mean curvature $H$ that forms a foliation of the complement of a compact subset of $M$. Moreover, they have  established a  characterization of the leaves of this foliation; see \cite[Theorem 5.1]{HuiskenYau}. As they explain in \cite[Theorem 4.2 and Remark 4.3]{HuiskenYau}, the foliation \eqref{cmc foliation} gives rise to both a canonical asymptotic coordinate system of $(M, g)$ and  a notion of geometric center of mass defined by $C_{CMC}=(C_{CMC}^1,\,C_{CMC}^2,\,C_{CMC}^3)$ where
\begin{align} \label{geometri com appendix} 
C_{CMC}^\ell=\lim_{H\to0}|\Sigma(H)|^{-1}\,\int_{\Sigma{(H)}}x^\ell\,\mathrm{d}\mu
\end{align} 
provided the limits on the right-hand side exist. \\ \indent   The results in \cite{HuiskenYau} have been improved in various directions. J.~Metzger has shown in \cite{Metzger} that the foliation \eqref{cmc foliation} exists under weaker regularity assumptions  on $(M,g)$.  The characterization of the leaves of the foliation has been strengthened by J.~Qing and G.~Tian  \cite{QingTian}. They show that the leaves are the only large stable constant mean curvature spheres that enclose the center of $(M, g)$. J.~Metzger and the first-named author have shown in \cite[Theorem 1.1]{EichmairMetzger} that the leaves are the only isoperimetric surfaces of their respective volume in $(M, g)$. In the case where the scalar curvature of $(M, g)$ is nonnegative, the characterization of the leaves has been refined further by J.~Metzger and the first-named author \cite{Eichmair-Metzer:2012}, by S.~Brendle and the first-named author \cite{BrendleEichmair}, by A.~Carlotto, O.~Chodosh, and the first-named author  \cite{CCE}, by O.~Chodosh and the first named-author  \cite{chodosh2017global,chodoshfar}, and by the authors \cite{acws}. In particular, if the scalar curvature of $(M,g)$ is nonnegative and satisfies a mild growth condition, then the leaves of the foliation \eqref{cmc foliation} are the only large closed stable constant mean curvature surfaces; see \cite[Theorem 50]{acws} and the discussion there. An alternative proof of the characterization result by J.~Qing and G.~Tian \cite{QingTian} in the case where the scalar curvature of $(M, g)$ is nonnegative has been given by O.~Chodosh and the first-named author in \cite[Appendix C]{chodosh2017global}. L.-H.~Huang has shown in \cite[Theorem 2]{Huang2} that the geometric center of mass \eqref{geometri com appendix} agrees with the Hamiltonian center of mass \eqref{center of mass} of $(M,g)$ provided either one of them exists. Moreover, C.~Cederbaum and C.~Nerz have shown that the existence of the limits in \eqref{geometri com appendix} requires additional assumptions to those stated in \cite{HuiskenYau}; see \cite[p.~1624]{CederbaumNerz}.  Finally, we mention that J.~Metzger and the first-named author \cite[Theorem 1.1]{EMinv} have shown that, also in the case where $n\geq 3$, there exists an asymptotic foliation of $(M,g)$ by stable constant mean curvature spheres, which are the unique solutions of the isoperimetric problem for the volume they enclose, and that the geometric center of mass \eqref{geometri com appendix} agrees with the Hamiltonian center of mass \eqref{center of mass} of $(M,g)$ provided either one of them exists. 

\section{Related results in asymptotically flat Riemannian three-manifolds}
\label{other methods appendix}
In this section, we give an overview of the strategies used in the works of L.-H.~Huang \cite{Huang},  of S.~Ma \cite{Ma}, and  of C.~Nerz \cite{Nerz}. We compare these strategies to the methods  used and developed in this paper. \\ \indent 
Let $(M,g)$ be a connected complete Riemannian three-manifold that is $C^3$-asymptotically flat.\\ \newline 
\textbf{The work of L.-H.~Huang} \cite{Huang}.
 L.-H.~Huang  has established Theorem \ref{existence thm nerz} under the additional assumption that $(M,g)$ satisfies the Regge-Teitelboim conditions in \cite{Huang}. To construct the asymptotic foliation by constant mean curvature spheres, L.-H.~Huang argues in two steps. First, the sphere $S_{\xi,\lambda}$ is perturbed to a surface $\tilde \Sigma_{\xi,\lambda}$ with $\operatorname{proj}_{\Lambda_0(S_{\xi,\lambda})^\perp}H(\tilde \Sigma_{\xi,\lambda})=O(\lambda^{-1-2\,\tau})$; see \cite[Lemma 3.2]{Huang}. This improves  the estimate \eqref{MC expansion}. The Regge-Teitelboim conditions ensure that certain error terms arising in this perturbation are  small; see \cite[(3.6)]{Huang}. Second, the spheres $\tilde \Sigma_{\xi,\lambda}$ are perturbed to obtain a constant mean curvature sphere $\Sigma(\lambda)$; see \cite[Theorem 3.1]{Huang}. L.-H.~Huang observes that this second perturbation is obstructed  unless $m\neq 0$ and  $\xi\approx \lambda^{-1}\,C$, where $C$ is the Hamiltonian center of mass; see \cite[(3.17)]{Huang}. We remark that \eqref{improved estimate} provides a geometric explanation for this obstruction. Note that Theorem \ref{com thm} specified to the case where $n=3$ follows as a corollary from the construction. \\ \indent To establish uniqueness, L.-H.~Huang first derives curvatures estimates for large stable constant mean curvature spheres $\Sigma\subset M$ satisfying \eqref{huang pinch} and proves a result similar to Lemma \ref{distance comp}; see \cite[Corollary 4.11]{Huang}. Using an argument analogous to but slightly coarser than that in the proof of Theorem \ref{uniqueness thm 2}, L.-H.~Huang obtains a preliminary estimate on the barycenter of $\Sigma$. This implies that $\Sigma$ is within the range of the  implicit function theorem. It then follows from local uniqueness that $\Sigma$ belongs to the asymptotic foliation; see \cite[\S4.4]{Huang}. \\ \newline 
\textbf{The work of S.~Ma} \cite{Ma}.   
The proof of Theorem \ref{Ma thm} by S.~Ma \cite{Ma} expands upon the method of  J.~Qing and G.~Tian in \cite{QingTian}. \\ \indent Let $(M,g)$ be asymptotically flat and suppose that $\{\Sigma_i\}_{i=1}^\infty$ is a sequence of stable constant mean curvature spheres $\Sigma_i\subset M$ with $\rho(\Sigma_i)\to\infty$ and $\rho(\Sigma_i)=o(\lambda(\Sigma_i))$ as $i\to\infty$. To begin with, S.~Ma notes that
\begin{align} \label{MA} 
0=\int_{\Sigma_i} (H-\bar H)\,\bar g(a,\bar\nu)\,\mathrm{d}\bar\mu
\end{align} 
for all $i$ and $a\in\mathbb{R}^3$ with $|a|=1$; see \cite[(34)]{Ma}. To obtain a contradiction, S.~Ma shows that the right-hand side of \eqref{MA} is essentially proportional to the mass of $(M,g)$ as $i\to\infty$. To this end, S.~Ma studies the contributions to the integral in \eqref{MA} in the three parts of the surface $\Sigma_i$ where 
\begin{itemize}
\item[$\circ$] $|x|\approx\rho(\Sigma_i)$,
\item[$\circ$] $|x|\approx\lambda(\Sigma_i)$,
\item[$\circ$] $\rho(\Sigma_i)\ll |x|\ll \lambda(\Sigma_i)$;
\end{itemize}
see \cite[\S5]{Ma}.
To obtain sufficient analytic control to estimate these contributions, S.~Ma uses the fact  that $\{\nu(\Sigma_i)\}_{i=1}^\infty$ forms a sequence of almost harmonic maps.
\\ \indent 
By comparison, in our proof of Theorem \ref{uniqueness thm} in Section \ref{sec:uniqueness}, we start with  identity \eqref{4 0}, which differs slightly from \eqref{MA}. Note that the right-hand side of \eqref{4 0} stands in for the derivative of the function $G_\lambda$. In view of Lemma \ref{g expansion}, we expect this derivative to be proportional to the mass of $(M,g)$. As shown in Proposition \ref{prop CE}, the assumption of nonnegative scalar curvature leads to strong analytic control of the surfaces $\Sigma_i$. This control, in conjunction with the integration by parts formula from Lemma \ref{ce int by parts new}, is sufficient to conclude the argument. In particular, we do not require the delicate analysis of the unit normal as an almost harmonic map in the region where $\rho(\Sigma_i)\ll |x|\ll \lambda(\Sigma_i)$ developed by S.~Ma \cite[\S 4]{Ma} building on the work of J.~Qing and G.~Tian \cite[\S 4]{QingTian}.\\ \newline 
\textbf{The work of C.~Nerz} \cite{Nerz}.
C.~Nerz's proof of Theorem \ref{existence thm nerz} in \cite{Nerz} expands upon the continuity argument developed by J.~Metzger in \cite{Metzger}. Given a $C^2$-asymptotically flat metric $g$ on $\{x\in\mathbb{R}^3:|x|>|m|/2\}$ with mass $m\neq 0$, C.~Nerz defines the family of metrics $\{g_t:t\in[0,1]\}$ where
\begin{align} \label{continuity} 
g_t=t\,g+(1-t)\,\bigg(1+\frac{m}{2\,|x|}\bigg)^4\,\bar g;
\end{align} 
see \cite[\S5]{Nerz}.
Note that $g_1=g$ and that $g_0$  is equal to Schwarzschild  with mass $m$. Let $I\subset[0,1]$ be the set of all $t\in[0,1]$ such that the conclusion of Theorem \ref{existence thm nerz} holds for  $g_t$. The difficult step when proving that $I=[0,1]$ is to show that $I$ is open. To this end, C.~Nerz notes that the stability operator \eqref{stability operator} of large, centered constant mean curvature spheres with nonvanishing Hawking mass is invertible, even if $(M,g)$ is only asymptotically flat; see \cite[Proposition 4.7]{Nerz}. In Lemma \ref{g expansion}, we provide a geometric explanation for this invertibility.  \\ \indent Moreover, using the same kind of continuity argument, C.~Nerz observes that the leaves of the foliation in Theorem \ref{existence thm nerz} are asymptotically symmetric provided $(M,g)$ satisfies the weak Regge-Teitelboim conditions \eqref{RT Nerz}. This step should be  compared to Lemma \ref{u odd lemma}. To establish Theorem \ref{com nerz thm}, C.~Nerz goes on to prove that 
$$
\frac{d}{dt}C_{CMC}(g_t)=C
$$
where $C_{CMC}(g_t)$ is the geometric center of mass \eqref{geometri com} of $(M,g_t)$ and $C$ is the Hamiltonian center of mass \eqref{center of mass} of $(M,g)$; see \cite[\S6]{Nerz}.
\\ \indent 
To conclude uniqueness in Theorem \ref{existence thm nerz}, C.~Nerz follows an idea developed by J.~Metzger. More precisely, one may reverse the continuity argument \eqref{continuity} and use the local uniqueness of the inverse function theorem and the global uniqueness of constant mean curvature spheres in Schwarzschild \cite{Brendle}; see \cite[\S5]{Nerz}.
\end{appendices}

\end{document}